\documentclass[11pt]{article}

\usepackage{comment,url,algorithm,algorithmic,graphicx,subcaption,relsize}
\usepackage{amssymb,amsfonts,amsmath,amsthm,amscd,dsfont,mathrsfs,mathtools,nicefrac}
\usepackage{float,psfrag,epsfig,color,xcolor,url,hyperref}
\usepackage{epstopdf,bbm,mathtools,enumitem}
\usepackage[toc,page]{appendix}
\usepackage[mathscr]{euscript}

\usepackage[top=1in, bottom=1in, left=1in, right=1in]{geometry}

\def\balign#1\ealign{\begin{align}#1\end{align}}
\def\baligns#1\ealigns{\begin{align*}#1\end{align*}}
\def\balignat#1\ealign{\begin{alignat}#1\end{alignat}}
\def\balignats#1\ealigns{\begin{alignat*}#1\end{alignat*}}
\def\bitemize#1\eitemize{\begin{itemize}#1\end{itemize}}
\def\benumerate#1\eenumerate{\begin{enumerate}#1\end{enumerate}}

\newenvironment{talign*}
 {\csname align*\endcsname}
 {\endalign}
\newenvironment{talign}
 {\csname align\endcsname}
 {\endalign}

\def\balignst#1\ealignst{\begin{talign*}#1\end{talign*}}
\def\balignt#1\ealignt{\begin{talign}#1\end{talign}}

\let\originalleft\left
\let\originalright\right
\renewcommand{\left}{\mathopen{}\mathclose\bgroup\originalleft}
\renewcommand{\right}{\aftergroup\egroup\originalright}

\def\tinycitep*#1{{\tiny\citep*{#1}}}
\def\tinycitealt*#1{{\tiny\citealt*{#1}}}
\def\tinycite*#1{{\tiny\cite*{#1}}}
\def\smallcitep*#1{{\scriptsize\citep*{#1}}}
\def\smallcitealt*#1{{\scriptsize\citealt*{#1}}}
\def\smallcite*#1{{\scriptsize\cite*{#1}}}

\def\reals{\mathbb{R}} %

\def\<{\left\langle} %
\def\>{\right\rangle}

\DeclareSymbolFont{rsfs}{U}{rsfs}{m}{n}
\DeclareSymbolFontAlphabet{\mathscrsfs}{rsfs}

\newcommand{\grad}{\boldsymbol \nabla}

\ifdefined\nonewproofenvironments\else
\ifdefined\ispres\else
\newtheorem{theorem}{Theorem}
\newtheorem{lemma}[theorem]{Lemma}
\newtheorem{corollary}[theorem]{Corollary}
\newtheorem{definition}[theorem]{Definition}

\renewenvironment{proof}{\noindent\textbf{Proof.}\hspace*{.3em}}{\qed\\}
\newenvironment{proof-sketch}{\noindent\textbf{Proof Sketch}
  \hspace*{1em}}{\qed\bigskip\\}
\newenvironment{proof-idea}{\noindent\textbf{Proof Idea}
  \hspace*{1em}}{\qed\bigskip\\}
\newenvironment{proof-of-lemma}[1][{}]{\noindent\textbf{Proof of Lemma {#1}}
  \hspace*{1em}}{\qed\\}
\newenvironment{proof-of-theorem}[1][{}]{\noindent\textbf{Proof of Theorem {#1}}
  \hspace*{1em}}{\qed\\}
 \newenvironment{proof-of}[1][{}]{\noindent\textbf{Proof of {#1}}
  \hspace*{1em}}{\qed\\}
\newenvironment{proof-attempt}{\noindent\textbf{Proof Attempt}
  \hspace*{1em}}{\qed\bigskip\\}

\newenvironment{remark}{\noindent\textbf{Remark.}
  \hspace*{0em}}{\smallskip}%

\fi

\newtheorem{proposition}[theorem]{Proposition}

\newtheorem{assumption}{Assumption}
\fi
\makeatletter
\@addtoreset{equation}{section}
\makeatother

\hypersetup{
  colorlinks,
  linkcolor={red!50!black},
  citecolor={blue!50!black},
  urlcolor={blue!80!black}
}

\newcommand{\eq}[1]{\begin{align}#1\end{align}}

\renewcommand{\Pr}[1]{\mathbb{P}\left( #1 \right)}

\usepackage[normalem]{ulem}
\usepackage{soul}

\usepackage{graphicx,xcolor,caption,enumitem}
\usepackage{wrapfig}

\usepackage{mathtools}

\renewcommand {\bf}{\mathbf}
\newcommand {\bm}{\boldsymbol}

\newcommand {\trsp}{^\mathsf {T}}

\newcommand {\Exp}{ \mathbb E }
\newcommand {\sg}{\operatorname{sgn}}
\renewcommand {\Pr}{ \mathbb P }
\renewcommand{\le}{\leqslant}
\renewcommand{\ge}{\geqslant}
\renewcommand{\d}{\mathrm d}

\newcommand{\step}{\gamma}
\newcommand{\noise}{\bxi}
\newcommand{\bx}{\boldsymbol x}
\newcommand{\barx}{\overline{\bm x}}
\newcommand{\bxi}{\boldsymbol \xi}
\newcommand{\sF}{\mathcal{F}}
\newcommand{\dist}{ \nu}

\newcommand{\sphere}{S}
\newcommand{\cone}{\mathbb S}

\usepackage{theoremref}
\usepackage{bbm}
\usepackage{natbib}
\mathtoolsset{showonlyrefs}

\makeatletter
\def\@fnsymbol#1{\ensuremath{\ifcase#1\or *\or \dagger\or \ddagger\or
   \mathsection\or \|\or **\or \dagger\dagger
   \or \ddagger\ddagger \else\@ctrerr\fi}}
    \makeatother

\begin{document}

\title{Convergence Rates of Stochastic Gradient Descent \\ under Infinite Noise Variance}

 \author{\!\!\!\!\!
 Hongjian Wang\thanks{
  Department of Computer Science and Technology at Tsinghua University \texttt{hongjian.wang@aol.com}
 },
 \  
 Mert G\"{u}rb\"{u}zbalaban\thanks{Department of Management Science and Information Systems at Rutgers Business School \texttt{mg1366@rutgers.edu}},
 \  
 Lingjiong Zhu\thanks{Department of Mathematics at Florida State University \texttt{zhu@math.fsu.edu}},
 \  
 Umut \c{S}im\c{s}ekli\thanks{INRIA - D\'{e}pt.\ d'Informatique de l'\'{E}cole Normale Sup\'{e}rieure - PSL Research University \texttt{umut.simsekli@inria.fr}},
 \ 
 Murat A. Erdogdu\thanks{
  Department of Computer Science and Department of Statistical Sciences at
  University of Toronto, and Vector Institute \texttt{erdogdu@cs.toronto.edu}
 }
}

\maketitle

\begin{abstract}
Recent studies have provided both empirical and theoretical evidence illustrating that heavy tails can emerge in stochastic gradient descent (SGD) in various scenarios. Such heavy tails potentially result in iterates with diverging variance, which hinders the use of conventional convergence analysis techniques that rely on the existence of the second-order moments. In this paper, we provide convergence guarantees for SGD under a state-dependent and heavy-tailed noise with a potentially infinite variance, for a class of strongly convex objectives. In the case where the $p$-th moment of the noise exists for some $p\in [1,2)$, we first identify a condition on the Hessian, coined `$p$-positive (semi-)definiteness', that leads to an interesting interpolation between positive semi-definite matrices ($p=2$) and diagonally dominant matrices with non-negative diagonal entries ($p=1$). Under this condition, we then provide a convergence rate for the distance to the global optimum in $L^p$. Furthermore, we provide a generalized central limit theorem, which shows that the properly scaled Polyak-Ruppert averaging converges weakly to a multivariate $\alpha$-stable random vector.
Our results indicate that even under heavy-tailed noise with infinite variance, SGD can converge to the global optimum without necessitating any modification neither to the loss function or to the algorithm itself, as typically required in robust statistics.
We demonstrate the implications of our results to applications such as linear regression and generalized linear models subject to heavy-tailed data.
\end{abstract}

\section{Introduction}

We consider the unconstrained minimization problem
\begin{align}\label{eq:min-f}
  \underset{\bx\in\reals^d}{\text{minimize}}\ f(\bx),
\end{align}
using the stochastic gradient descent (SGD) algorithm. Initialized at $\bx_0 \in\reals^d$,
the SGD algorithm is given by the iterations,
\begin{equation} \label{eq:sgd}
  \bx_{t+1} = \bx_{t} - \step_{t+1} \left( \grad f(\bx_{t}) + \bxi_{t+1} (\bx_t)\right),\ \ t= 0,1,2,...
\end{equation}
where $\{\step_t\}_{t\in \mathbb N^+}$ denotes the step-size sequence,
and  $\{\noise_t\}_{t\in \mathbb N^+}$ is a martingale difference sequence adapted to a filtration $\{\sF_t\}_{t\in \mathbb N}$, characterizing the noise in the gradient (the sequence $\{\bx_t\}_{t\in \mathbb N}$ is also adapted to
the same filtration, if we assume $\bx_0$ is $\sF_0$-measurable).
Our focus is on the case where the noise is state dependent, 
and its variance is infinite, i.e., $\Exp{\big[\|\bxi_t\|_2^2\big]}=\infty$.

Many problems in modern statistical learning can be written in the form \eqref{eq:min-f},
where
$f(\bx)$ typically corresponds to the population risk, that is, $f(\bx) \coloneqq \Exp_{z\sim \dist}[\ell(\bx,z)]$ for a given loss function $\ell$ and an unknown data distribution $\dist$.
In practice, one observes independent and identically distributed (i.i.d.) samples $z_i\sim \dist$ for $i\in[n]$,
and estimates the population gradient $\grad f(\bx)$ with a noisy gradient at each iteration,
which is based on an empirical average over a subset of the samples $\{z_i\}_{i\in[n]}$.
Due to its simplicity, superior generalization performance, and well-understood theoretical guarantees,
SGD has been the method of choice for minimization problems arising in statistical machine learning.

Starting from the pioneering works of~\cite{robbins1951stochastic, chung1954stochastic, sacks1958asymptotic, fabian1968asymptotic, ruppert1988efficient, shapiro1989asymptotic,polyak1992acceleration},
theoretical properties of the SGD algorithm and its variants
have been receiving a growing attention under different scenarios.
Recent works, for example~\cite{tripuraneni2018averaging, su2018statistical, duchi2016local, toulis2017asymptotic,fang2018online, anastasiou2019normal, yu2020analysis}
establish convergence rates for SGD in various settings, and build on the analysis of~\cite{polyak1992acceleration} to prove a central limit theorem (CLT) for the
Polyak-Ruppert averaging,
which leads to novel methodologies to compute
confidence intervals using SGD. 
However, a recurring assumption in this line of work is the finite noise variance,
which may be violated frequently in modern frameworks.

Heavy-tailed behavior in statistical methodology may naturally arise from the underlying model,
or through the iterative optimization algorithm used during model training.
In robust statistics, one often encounters heavy-tailed noise behavior in data, 
which in conjunction with standard loss functions leads to infinite noise variance in SGD.
Very recently, heavy-tailed behavior is shown to emerge from the multiplicative noise in SGD,
when the step-size is large and/or the batch-size is small~\citep{hodgkinson2020multiplicative,gurbuzbalaban2020heavy}.
On the other hand, there is strong empirical evidence 
in modern machine learning that
the gradient noise often exhibits a heavy-tailed behavior, which indicates an infinite variance. For example, this is observed in fully connected and convolutional neural networks~\citep{simsekli19a,gurbuzbalaban2020fractional} as well as recurrent neural networks~\citep{zhang2019adam}. 
Thus, understanding the behavior of SGD under infinite noise variance becomes extremely important for at least two reasons. A \emph{computational complexity reason:}
modern machine learning and robust statistics frameworks lead to heavy-tailed behavior in SGD;
thus, understanding the performance of this algorithm in terms of precise convergence rates as well as the required conditions on the step-size sequence as a function of the `heaviness' of the tail become crucial in this setup.
A \emph{statistical reason:}
many inference methods that rely on Polyak-Ruppert averaging utilize
a CLT that holds under finite noise variance (see e.g.
online bootstrap and variance estimation approaches~\citep{fang2018online, su2018statistical,chen2020statistical}).
Using the same methodology in the aforementioned modern framework (under heavy-tailed noise) will ultimately result in incorrect confidence intervals, jeopardizing the statistical procedure.
Thus, establishing the limit distribution in this setting is of great importance.

In this work, we study the behavior of the SGD algorithm with diminishing step-sizes for a class of strongly convex problems when the noise variance is infinite. We establish the convergence rates of the SGD iterates towards the global minimum, and identify a sufficient condition on the Hessian of $f$, which interpolates positive semi-definiteness and diagonal dominance with non-negative diagonal entries.
We further study the Polyak-Ruppert averaging of the SGD iterates,
and show that the limit distribution is a multivariate $\alpha$-stable distribution. 
We illustrate our theory on linear regression and generalized linear models, demonstrating how to verify the conditions of our theorems. 
Perhaps surprisingly, our results show that even under heavy-tailed noise with infinite variance, SGD with diminishing step-sizes can converge to the global optimum without requiring any modification neither to the loss function or to the algorithm itself, as opposed to the conventional techniques used in robust statistics~\citep{huber2004robust}.
Finally, we argue that our work has potential implications in constructing confidence intervals in
the infinite noise variance setting.

\section{Preliminaries and Technical Background}
\par \textbf{Notational Conventions.}
By $\mathbb N$, $\mathbb N^+$ and $\mathbb R$ we denote the set of non-negative integers, positive integers, and real numbers respectively. For $m\in\mathbb N^+$, we define $[m]=\{1, \ldots, m\}$. We use italic letters (e.g. $x, \xi$) to denote scalars and scalar-valued functions, bold face italic letters (e.g. $\bm x, \bm \xi$) to denote vectors and vector-valued functions, and bold face upper case letters (e.g. $\bf A$) to denote matrices. We use $| \bm x |$ and $\| \bm x \|_p$ to denote the 2-norm and $p$-norm of a vector $\bm x$; $\|\bf A\|$ and $\| \bf A \|_p$ the operator 2-norm and operator $p$-norm of a matrix $\bf A$. The transpose of a matrix $\bf A$ and a vector $\bm x$ (viewed as a matrix with 1 column) are denoted by $\bf A\trsp$ and $\bm x \trsp$. 
If $\{\bf A_i\}_{i\in\mathbb N}$ is a sequence of matrices and $k>\ell$, the empty product $\prod_{i=k}^\ell \bf A_i$ is understood to be the identity matrix $\bf I$.
The asymptotic notations are defined in the usual way: for two sequences of real numbers $\{a_t\}_{t\in\mathbb N}$, $\{b_t\}_{t\in\mathbb N}$, we write $a_t = \mathcal O(b_t)$ if $\limsup_{t\to\infty} |a_t|/|b_t| < \infty$, $a_t = o(b_t)$ if $\limsup_{t\to\infty} |a_t|/|b_t| = 0$, $a_t = \Theta(b_t)$ if both $a_t = \mathcal O(b_t)$ and $b_t = \mathcal O(a_t)$ hold, and $a_t \asymp b_t$ if $\lim_{t\to\infty} |a_t|/|b_t|$ exists and is in $(0,\infty)$. If $a_t = \mathcal{O}(b_t t^\varepsilon)$ for any $\varepsilon>0$, we say $a_t = \tilde{\mathcal{O}}(b_t)$. Sufficiently large or sufficiently small positive constants whose values do not matter are written as $C, C_0, C_1, \ldots$, sometimes without prior introduction.
If $\bm X_1, \bm X_2, \ldots$ is a sequence of random vectors taking value in $\mathbb R^n$ and $\mu$ is a probability measure on $\mathbb R^n$, we write $\bm X_t \xrightarrow[t\to \infty]{\mathcal D} \mu $
if $\{\bm X_t\}_{t\in\mathbb N^+}$ converges in distribution (also called `converges weakly') to $\mu$.

\vspace{4pt}

\noindent\textbf{Stochastic Approximation.}
In the SGD recursion \eqref{eq:sgd}, we can replace $\grad f$ with arbitrary continuous function $\bm R:\mathbb R^n \to \mathbb R^n$, and consider the same iterations that stochastically approximate the zero $\bm x^*$ of $\bm R$,
\begin{equation}\label{eqn:SA}
    \bx_{t+1} = \bx_{t} - \step_{t+1} \left( \bm R(\bx_{t}) + \bxi_{t+1} (\bx_t)\right).
\end{equation}
This is called the \emph{stochastic approximation} process \citep{robbins1951stochastic}, which is a predecessor of stochastic gradient descent and describes a larger family of iterative algorithms (\cite{kushner2003stochastic}, Chapters 2 and 3). Theoretical investigation into recursion \eqref{eqn:SA} has been active ever since its invention, especially under finite noise variance assumption: \cite{robbins1951stochastic} prove the recursion \eqref{eqn:SA} can lead to the $L^2$ convergence $\lim_{t\to\infty}\Exp[|\bm x_t - \bm x^*|^2]=0$; \cite{chung1954stochastic} further calculates an exact convergence rate (see \eqref{eqn:chung-fv-rate}); \cite{blum1954approximation} presents an elegant proof that the convergence of $\bx_t$ to $\bm x^*$ can hold almost surely. The asymptotic distribution of \eqref{eqn:SA} is also the discovery of \cite{chung1954stochastic}, the Theorem 6 of which states that $\gamma_t^{-1/2}(\bx_t - \bx^*)$ converges weakly to a normal distribution; \cite{polyak1992acceleration} and \cite{ruppert1988efficient} independently introduce the concept of `averaging the iterates',
\begin{equation}
\label{eqn:sa_intro}
    \overline{\bx}_t = \frac{\bx_0 + \ldots + \bx_{t-1}}{t},
\end{equation}
showing the striking result that $\sqrt t (\overline{\bx}_t - \bx^*)$ converges weakly to a fixed normal distribution \emph{regardless of the choice of the step-size $\{\gamma_t\}_{t\in \mathbb N^+}$}.
Recently, optimization algorithms that can handle heavy-tailed $\bxi$ have been proposed \citep{davis2019low,nazin2019algorithms,gorbunov2020stochastic}; however, they still rely on a \emph{uniformly} bounded variance assumption, hence do not cover our setting.

\par Compared with the copious collection of theoretical studies on stochastic approximation with finite variance mentioned above, papers on \emph{infinite} variance stochastic approximation are extremely scarce, and we shall summarize the only four papers known to us to the best of our knowledge. \cite{krasulina1969stochastic} is the first to consider such problems, proving almost sure and $L^p$ convergence for the one-dimensional stochastic approximation process without variance. The weak convergence of the iterates (without averaging) $t^{1/\alpha}(\bx_t - \bx^*)$ is also considered by \cite{krasulina1969stochastic}, but only for the fastest-decaying step-size $\gamma_t = 1/t$. \cite{goodsell1976almost} discuss how $\bx_t \to \bx^*$ in probability can imply $\bx_t \to \bx^*$ almost surely, when no finite variance is assumed, and \cite{li1994almost} provides a necessary and sufficient condition for almost sure convergence of $\bx_t \to \bx^*$, stating that faster-decaying step-size $\gamma_t = o(t^{-1/p})$ is required when moments of lower orders $\Exp[|\bm \xi_t|^p]$ are not in place. \cite{anantharam2012stochastic} show that although step-size that decays slower than $t^{-1/p}$ cannot yield almost sure convergence, $L^p$ convergence can still hold under what they call the `stability assumption', but their analysis technique provides no convergence rate.
Recently, \cite{csimcsekli2019heavy} and \cite{zhang2019adam} considered SGD with heavy-tailed $\bxi$ having \emph{uniformly bounded} $p$-th order moments. Besides not being able to handle state-dependent noise due to this uniform moment condition, \cite{csimcsekli2019heavy} imposed further conditions on $\bm R = \nabla f$ such as global H\"{o}lder continuity for a non-convex $f$, whereas \cite{zhang2019adam} modified SGD with `gradient clipping', in order to be able to compensate the effects of the heavy-tailed noise.

\par Finally, we shall mention that a class of stochastic recursions similar to \eqref{eqn:SA} have been considered in the dynamical systems theory \citep{mirek2011heavy,buraczewski2012asymptotics,buraczewski2016stochastic}, for which generalized central limit theorems with $\alpha$-stable limits have been proven. However, such techniques typically require $\bm R$ to be (asymptotically) linear and the step-sizes to be constant as they heavily rely on the theory of time-homogeneous Markov processes. Hence, their approach does not readily generalize to the setting of our interest, i.e., non-linear $\bm R$ and diminishing step-sizes, where the latter is crucial for ensuring convergence towards the global optimum.

\vspace{4pt}

\noindent\textbf{Stable Distributions.} 
In probability theory, a random variable $X$
is \emph{stable} if its distribution is non-degenerate
and satisfies the following property: Let $X_{1}$ and $X_{2}$ be independent copies of $X$.
Then, for any constants $a,b>0$, the random variable
$aX_{1}+bX_{2}$ has the same distribution as $cX+d$ 
for some constants $c>0$ and $d$ (see e.g. \citep{ST1994}).
The stable distribution is also referred to as
the $\alpha$-stable distribution, first proposed
by \cite{paul1937theorie}, where $\alpha \in (0,2]$ denoting the stability parameter. The case $\alpha=2$ corresponds to the normal distribution,
and the variance under this distribution is undefined for any $\alpha<2$.
The multivariate $\alpha$-stable distribution dates back to \cite{Feldheim}, 
which is a multivariate generalization of the univariate $\alpha$-stable distribution, which is also uniquely characterized 
by its characteristic function. 
In particular, an $\mathbb{R}^{d}$-valued random vector $X$ 
has a multivariate $\alpha$-stable distribution, denoted
as $\bm X\sim\mathcal{S}(\alpha,\Lambda,\delta)$ if
the joint characteristic function of $\bm X$ is given by
\begin{equation}\label{char:alpha:stable:multi}
\Exp\left[\exp\left(i\bm u\trsp\bm X\right)\right]
=\exp\Big\{-\int_{\bm s\in S_{2}}(|\bm u\trsp\bm s|^{\alpha}+i\nu(\bm u\trsp\bm s,\alpha))\Lambda(\d\bm s)+i\bm u\trsp\delta\Big\},
\end{equation}
for any $\bm u\in\mathbb{R}^{d}$, and $0<\alpha\le 2$.
Here, $\alpha$ is the tail-index, $\Lambda$ is a finite measure on $S_{2}$ known as the spectral measure, $\bm \delta\in\mathbb{R}^{d}$ is a shift vector, 
and $\nu(y,\alpha):=-\sg(y)\tan(\pi\alpha/2)|y|^{\alpha}$
for $\alpha\neq 1$ and $\nu(y,\alpha):=(2/\pi)y\log|y|$ for $\alpha=1$ for any $y\in\mathbb{R}$, 
and $S_{2}$ denotes the unit sphere in $\mathbb{R}^{d}$; i.e. $S_{2}=\{\bm s\in\mathbb{R}^{d}:\Vert\bm s\Vert_{2}=1\}$. 
Stable distributions also appear as the limit
in the Generalized Central Limit Theorem (GCLT) \citep{gnedenko1968limit}, which states that for a sequence of i.i.d.\ random variables whose distribution
has a power-law tail with index $0< \alpha <2$, 
the normalized sum converges to an $\alpha$-stable distribution as the number of summands grows.
\vspace{4pt}

\noindent\textbf{Domains of Normal Attraction of Stable Distributions.}
Let $\bm X_{1},\bm X_{2},\ldots,\bm X_{n}$ be an i.i.d. sequence
of random vectors in $\mathbb{R}^{d}$ with a common distribution function $F(\bm x)$.
If there exists some constant $a>0$ and a sequence $b_{n}\in\mathbb{R}^{d}$ such that
\begin{equation}\label{eqn:a}
\frac{\bm X_{1}+\cdots+\bm X_{n}}{an^{1/\alpha}}-b_{n}
\xrightarrow[n\to \infty]{\mathcal D} \mu,
\end{equation}
then
$F(\bm x)$ is said to belong to the \emph{domain of normal attraction} of the law $\mu$, 
and $\alpha$ is the characteristic exponent of the law $\mu$ \cite[page 181]{gnedenko1968limit}.
If $\mu$ is an $\alpha$-stable
distribution, then we say $F(\bm x)$ is said to belong to the domain of normal attraction of an $\alpha$-stable distribution.
For example, the Pareto distribution
belongs to the domain of normal attraction of an $\alpha$-stable law.
In Appendix~\ref{sec:more-prelim}, we provide more details as well as a sufficient and necessary condition
for being in the domain of normal attraction of an $\alpha$-stable law.

\section{Convergence of SGD under Heavy-tailed Gradient Noise}
\label{sec:convergence}
In this section, we identify sufficient conditions for the convergence of SGD under heavy tailed gradient noise, and derive the explicit rate estimates.
In the standard setting when the noise variance is finite,
some notion of positive definite Hessian assumption is frequently utilized to achieve convergence (see for example ~\cite{polyak1992acceleration,tripuraneni2018averaging, su2018statistical, duchi2016local, toulis2017asymptotic,fang2018online,anastasiou2019normal}).
When the noise variance is infinite, but it has finite $p$-th moment for $p \in [1,2)$,
one requires a stronger notion of positive definiteness on the Hessian,
which leads to an interesting interpolation between the positive semi-definite cone (as $p\to 2$), 
and the cone of diagonally dominant matrices with non-negative diagonal entries ($p=1$).

\subsection{$p$-Positive Definiteness}

First, we introduce a signed power of vectors which will be used when defining a family of matrices.

\begin{wrapfigure}{r}{0.37\textwidth}
\centering
\includegraphics[width=0.37\textwidth]{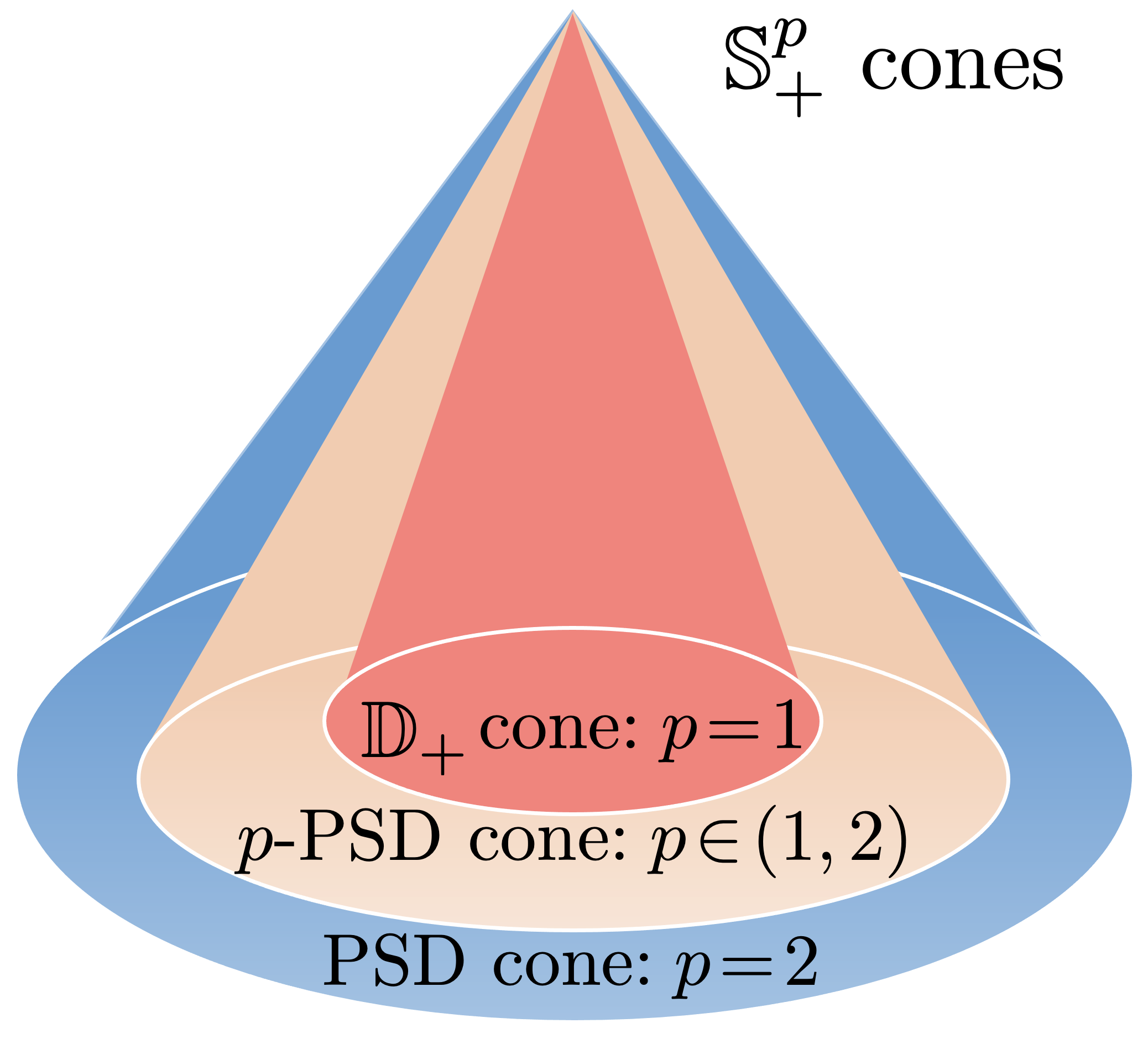}
\vspace{-15pt}
\caption{\small\!\! Geometry of $p$-PSD matrices.
$\mathbb D_+$ cone refers to the cone of diagonally dominant matrices with non-negative diagonal entries. Their inclusion relationship is given in Propositions~\ref{rmk:1pd=udd} and \ref{rmk:ppd->2pd}.} 
\label{fig:ppos}
\vspace{-15pt}
\end{wrapfigure}
For $\bm v = (v^1, \ldots, v^n) \trsp \in \mathbb R^n$ and $q \ge 0$, we let
\begin{equation}\label{eqn:signed-power}
    \bm v^{\langle q \rangle} = \left(\sg\left(v^1\right) \left|v^1\right|^q, \ldots, \sg\left(v^n\right)\left|v^n\right|^q  \right) \trsp. 
\end{equation}
Denoting the $n$-dimensional $\ell_p$ unit sphere with $\sphere_p = \{\bm v \in \reals^n  : \|\bm v\|_p=1 \}$,
and the set of $n\times n$ symmetric matrices with $\mathbb S$,
we now define the following subset of $\mathbb S$.

\begin{definition}[$p$-positive definiteness]\thlabel{def:p+}
Let $p\ge 1$ and $\bf Q$ be a symmetric matrix. We say that $\bf Q$ is $p$-positive definite if 
for all $\bm v \in \sphere_p$, $\bm v \trsp \bf Q \bm v^{\langle p - 1 \rangle} > 0$.
Similarly, we call $\bf Q$ $p$-positive semi-definite if
for all $\bm v \in \sphere_p$, $\bm v \trsp \bf Q \bm v^{\langle p - 1 \rangle} \ge 0$.
\end{definition}

It is not hard to see that the set of $p$-positive semi-definite matrices ($p$-PSD) defines a closed pointed cone, which we denote by $\mathbb S^p_+$, 
with interior as the set of $p$-positive definite matrices ($p$-PD),
denoted by $\cone^p_{++}$.
We are mainly interested in the case $1\le p<2$. Note that
$\cone^2_{+}$ coincides with the standard PSD cone,
and we show in Section~\ref{sec:p+} that $\cone^1_{+}$ is exactly the cone of diagonally dominant matrices with non-negative diagonal entries,
denoted by $\mathbb D_+$. For any $p \in [1,2]$, these cones satisfy the following
\eq{
\mathbb D_+ = \cone^1_+ \subseteq \cone^p_+ \subseteq \cone^2_+.
}
Figure~\ref{fig:ppos} is an hypothetical illustration of the relationships between these cones.

For a uniform version of Definition~\ref{def:p+}, we recall that every operator norm $\|\cdot\|_p$ induces the same topology on the set of $n$-dimensional matrices, which is just the usual topology on $\mathbb R^{n\times n}$. Further, the set of symmetric matrices $\cone$, as the set of zeros of the continuous function $\bf X \mapsto \bf X - \bf X \trsp$, is a closed set. Hence for a set $\mathcal M \subseteq \cone$, denoting its topological closure with $\overline{\mathcal{M}}$, we also have $\overline{\mathcal{M}} \subseteq\cone$. We are interested in the case where $\mathcal{M}$ is bounded.

\begin{definition}[uniform $p$-PD]\thlabel{def:unfmp+}
Let $p\ge 1$ and $\mathcal M \subset \mathbb S$ be a non-empty set of symmetric matrices. We say that $\mathcal M$ is uniformly $p$-PD if 
for all $\bf Q \in \overline{\mathcal M}$, 
we have $\bf Q \in \cone^p_{++}$.
\end{definition}
\par Notice that $\mathcal M$ is uniformly 2-PD if and only if the eigenvalues of the symmetric matrices in the set $\mathcal{M}$ are all lower bounded by a positive real number. Notice also that a finite subset of symmetric matrices is uniformly $p$-PD if and only if each element of the set is $p$-PD.

$p$-PSD cone emerges naturally when analyzing SGD algorithm in the heavy-tailed setting, interpolating between the standard PSD cone
to the cone of diagonally dominant matrices with non-negative diagonal entries. To the best of our knowledge, we are the first to study such families of matrices and their application in stochastic optimization. For further details about these cones, we refer interested reader to Appendix~\ref{sec:p+}.

We make the following uniform smoothness and the curvature assumptions on the Hessian of the objective function.
\begin{assumption}\label{as:ppos}
The set of matrices $\{ \grad^2 f(\bx): \bm x \in \mathbb R^n \}$ is bounded and uniformly $p$-PD.
\end{assumption}

\subsection{Rate of Convergence in $L^p$}

We fix a probability space $(\Omega, \mathcal F, \Pr)$ with filtration $\{\mathcal F_t\}_{t\in \mathbb N}$, and make the following assumption on the gradient noise sequence.

\begin{assumption}\label{as:noise}
Let $\bx_0$ be $\sF_0$-measurable. The gradient noise sequence $\{\bxi_t\}_{t\in \mathbb N^+}$ is given as 
\begin{equation}\label{eqn:noise-decomp}
    \bxi_{t+1}(\bx_t) =  \bm m_{t+1}(\bx_t)+ \bm \zeta_{t+1},
\end{equation}
 where $\{ \bm \zeta_t \}_{t\in \mathbb N^+}$ is an i.i.d. sequence with $\Exp[\bm \zeta_t] = 0$,
and $\Exp[|\bm \zeta_t|^p] < \infty$ for some $p$, and $\{ \bm m_t \}_{t\in \mathbb N^+}$ is a martingale difference sequence, and both sequences are adapted
 to the filtration 
 $\{\mathcal F_t\}_{t\in \mathbb N}$.
 
 Further, the state dependent component of the noise satisfies, for some $K>0$,
 \begin{equation}\label{eqn:state-dep-var}
     \Exp\left[ \left|\bm m_{t+1}(\bx_t)\right|^2  \mid  \mathcal F_t\right] \le K\left(1+|\bx_t|^2\right).
 \end{equation}
\end{assumption}

We note that the above assumptions also imply that both 
the gradient noise sequence $\{\bxi_t\}_{t\in \mathbb N^+}$ as well as the SGD iterates $\{\bx_t\}_{t\in \mathbb N}$
are adapted to the same filtration $\{\mathcal F_t\}_{t\in \mathbb N}$.
 We call $\bm m_t$ the \emph{state-dependent component} of the gradient noise,
 which naturally has a state dependent conditional second moment. The variance of this component of the noise can be arbitrarily large depending on the state; yet, 
 for a given state $\bx_t$, it is guaranteed to be finite.
 The \emph{heavy-tailed} noise behavior is due to $\bm \zeta_t$, which may have an infinite variance for $p<2$ (i.e., the second moment is not undefined).
 We point out that such a decomposition \eqref{eqn:noise-decomp} arises in
many instances of stochastic approximation subject to heavy-tailed noise with long-range dependencies and has been considered in the literature, see e.g. \citet{polyak1992acceleration} and \citet{anantharam2012stochastic}. We shall show in Section~\ref{sec:examples} that such noise structure arises in practical applications such as linear regression and generalized linear models subject to heavy-tailed data.

Our first result provides a convergence rate in $L^p$, for the SGD algorithm to the unique minimizer $\bx^*$ of the objective function $f$ with uniformly $p$-PD Hessian, when the noise sequence $\{\bm \xi_t\}_{t \in \mathbb N^+}$ has potentially an infinite variance. 
\begin{theorem}\thlabel{thm:rate-lp} Suppose Assumptions~\ref{as:ppos} and \ref{as:noise} hold for some $1 < p\le 2$. For step-size satisfying $\gamma_t \asymp t^{-\rho}$ with $\rho \in (0, 1)$,
the error of the SGD iterates $\{ \bx_t\}_{t\in \mathbb N}$ from the minimizer $\bx^*$ satisfies
\begin{equation}\label{eqn:lpweak}
\Exp\left[|\bx_t - \bx^*|^p\right] = \mathcal{O}\left(t^{ -\rho (p-1)  }  \right).
\end{equation}
Consequently, we have $\sup_{t\in \mathbb N^+}\Exp[|\bm \xi_t|^p] < \infty$.
\end{theorem}
The proof of \thref{thm:rate-lp} is provided in Appendix~\ref{sec:proofs}.
We observe that the convergence rate of SGD depends on the highest finite moment $p$ of the noise sequence, and faster rates are achieved for larger values of $p$. 
The fastest rate implied by our result is near $\mathcal{O}\left(t^{ -p+1}  \right)$, which is achieved for $\rho\approx1$; yet, SGD converges even for very slowly decaying step-size sequences with $\rho$ closer to $0$.

If the noise has further integrability properties with a finite $p$-th moment for all $p \in [q,\alpha)$ for some $q<\alpha$ and if uniform $p$-PD assumption holds, then faster rates are achievable. In particular, the following result is a consequence of \thref{thm:rate-lp}, and its proof is provided in Appendix~\ref{sec:proofs}.

\begin{corollary}\thlabel{thm:rate-lpalpha}
 For constants $q,\alpha$ satisfying $1 < q < \alpha \le 2$, suppose that Assumptions~\ref{as:ppos} and \ref{as:noise} hold for every $p \in [q, \alpha)$. For step-size satisfying $\gamma_t \asymp t^{-\rho}$ with $\rho \in (0, 1)$,
 the error of the SGD iterates $\{ \bx_t\}_{t\in \mathbb N}$ from the minimizer $\bx^*$ satisfies

\begin{equation}\label{eqn:lpstrong}
\Exp\left[|\bm x_t - \bx^*|^q\right] =  \tilde{\mathcal{O}} \left(t^{-\rho q \frac{\alpha - 1}{\alpha}}  \right). 
\end{equation}
\end{corollary}

\begin{remark}
The additional integrability assumption yields faster rates for any feasible step-size sequence since $p(\alpha-1)/\alpha \ge p-1$ for $p\in(1,2]$.
\end{remark}

\par Let us briefly compare our results stated above to those in the setting where the noise sequence has a finite variance. A classical convergence result that goes back to \citet[Theorem~5]{chung1954stochastic}\footnote{This result, like many other similar studies in the 1950s, concerns only the one-dimensional case. But they generalize easily to higher dimensions.} states that
\begin{equation}\label{eqn:chung-fv-rate}
  \Exp\left[|\bx_t - \bx^*|^r\right] = \Theta\left(t^{-\rho r/2}\right),
\end{equation}
where $r\ge2$ is an integer such that the $r$-th moment exists for the stochastic approximation process, and this
is achieved for strongly convex objective functions in one dimension (whose second derivative $\{ f''(\bx): \bm x \in \mathbb R \}$ satisfies the uniformly $2$-PD property)
with a step-size choice $\gamma_t \asymp t^{-\rho}$ for some $\rho \in (1/2,1)$. We point out that our rate \eqref{eqn:lpstrong} recovers the rate implied by \eqref{eqn:chung-fv-rate} when $r=2$,
 and extends it further to the case $1\le r<2$.

\section{Stable Limits for the Polyak-Ruppert Averaging}
\label{sec:averaging}
\par In this section, we establish the limit distribution of the Polyak-Ruppert averaging under infinite noise variance, extending
the asymptotic normality result given by \cite{polyak1992acceleration}
to $\alpha$-stable distributions. Let us fix an $\alpha \in (1, 2]$ and assume the following throughout this subsection.

\begin{assumption}\label{as:dona}
Let $\bx_0$ be $\sF_0$-measurable. The gradient noise sequence $\{\bxi_t\}_{t\in \mathbb N^+}$ is given as 
\begin{equation}\label{eqn:noise-decomp-2}
    \bxi_{t+1}(\bx_t) =  \bm m_{t+1}(\bx_t)+ \bm \zeta_{t+1},
\end{equation}
 where $\{ \bm \zeta_t \}_{t\in \mathbb N^+}$ is an i.i.d. sequence with $\Exp[\bm \zeta_t] = 0$,
 and it is in the domain of normal attraction of an $n$-dimensional symmetric $\alpha$-stable distribution $\mu$, i.e.,
 \begin{equation}\label{eqn:noise-attracted}
    \frac{\bm \zeta_1 + \ldots + \bm \zeta_t}{t^{1/\alpha}}  \xrightarrow[t\to \infty]{\mathcal D} \mu.
\end{equation}
The state dependent component $\{ \bm m_t \}_{t\in \mathbb N^+}$ is a martingale difference sequence
with a second-moment satisfying~\eqref{eqn:state-dep-var}, 
and both sequences are adapted
 to the filtration 
 $\{\mathcal F_t\}_{t\in \mathbb N}$.
\end{assumption}
The above assumption also implies that $\Exp[|\bm \zeta_t|^p] < \infty$ for every $p\in[1,\alpha)$, i.e., the moment condition on the i.i.d. heavy-tailed component of the noise in Assumption~\ref{as:noise} holds for every $p\in[1,\alpha)$.
\par Denoting the Polyak-Ruppert averaging by $\barx_t \coloneqq \frac{1}{t} (\bx_0+...+\bx_{t-1})$, we are interested in the asymptotic behavior of
\begin{equation}
    t^{1-1/\alpha} (\barx_t - \bx^*) =  \frac{(\bm x_0  + \ldots + \bm x_{t-1}) - t \bm x^*}{t^{1/\alpha}},
\end{equation}
for $\alpha \in (1,2]$. In the special case when $\alpha = 2$, it is known that this limit converges
to a multivariate normal distribution (which is a 2-stable distribution), a result proven in the seminal work by \cite{polyak1992acceleration}. Similarly, we begin with a result that considers a quadratic objective where the function $\grad f(\bx)$ is linear in $\bx$, and then building on this result, we establish the limit distribution of Polyak-Ruppert averaging also in the more general non-linear case.

\begin{theorem}[linear case]\thlabel{thm:stab-linear} Suppose the function $\grad f(\bx)$ is affine, i.e. $\grad f (\bm x) = \bf A \bm x - \bm b$ for a real matrix $\bf A \in \reals^{n\times n}$ and a real vector $\bm b \in \reals^n$ and there exist scalars $p,\rho$ satisfying
\begin{align}
    \max \left( \frac{\alpha + \alpha \rho}{1 + \alpha \rho}, \alpha\rho \right) \le p\le \alpha,
\end{align}
such that $\bf A$ is $p$-PD and $\rho \in (0, 1)$. 
If the noise sequence satisfies Assumption~\ref{as:dona},
then
for the step-size satisfying $\gamma_t \asymp t^{-\rho}$,
the normalized average $t^{1-1/\alpha} (\barx_t - \bx^*)$ converges weakly to an $n$-dimensional $\alpha$-stable distribution.
\end{theorem}

We observe from the above theorem that $\alpha$-stable limit is achieved for Polyak-Ruppert averaging for any step-size sequence with index $\rho \in (0,1]$. Thus, in the linear case,
the size of the interval of feasible indices is the same in both heavy- and light-tailed noise settings (see e.g.
\cite{polyak1992acceleration} and \cite{ruppert1988efficient}).
Notably, $\alpha$-stable limit of the averaged iterates does not depend on the index $\rho$.
Non-asymptotic rates are required to see the effect of step-size more clearly.

The next result generalizes Theorem~\ref{thm:stab-linear} to the setting where $\grad f(\bx)$ is non-linear.
\begin{theorem}[non-linear case]\thlabel{thm:stab-nonlin} Let $1 < 1/\rho < q < \alpha$ and suppose Assumption~\ref{as:ppos} holds for every $p\in[q,\alpha)$. Assume further that the gradient $\grad f(\bm x)$ can be approximated using the Hessian matrix $\bf \grad^2 f(\bx^*)$ around the minimizer $\bm x^*$ as
\begin{align}\label{eq:local-linearity}
  \left|\grad f(\bm x) - \bf \grad^2 f(\bx^*) (\bm x - \bm x^*)  \right| \le K \left|\bm x - \bm x^*\right|^q.
\end{align}
If the noise sequence satisfies Assumption~\ref{as:dona},
for the step-size satisfying $\gamma_t \asymp t^{-\rho}$, the normalized average $t^{1-1/\alpha} (\barx_t - \bx^*)$ converges weakly to an $n$-dimensional $\alpha$-stable distribution.
\end{theorem}
The additional assumption~\eqref{eq:local-linearity} is standard (see e.g. \citet[Assumption~3.2]{polyak1992acceleration}), which simply imposes a linearity condition on the gradient of $f$
with an order-$q$ polynomial error term. We notice that the size of the interval of feasible indices
$\rho \in (1/\alpha, 1)$ is smaller this time compared to the light tailed case, where
\citet[Theorem~2]{polyak1992acceleration} allows $\rho \in (1/2, 1)$.

The above theorem establishes that,
when the noise has diverging variance,
the Polyak-Ruppert averaging admits an $\alpha$-stable limit rather than a standard CLT. This result has potential implications in statistical inference in the presence of heavy-tailed data. Inference procedures that take into account the computational part of the training procedure (instead of drawing conclusions for the minimizer of the empirical risk) rely typically on variations of Polyak-Ruppert averaging and the CLT they admit~\citep{fang2018online, su2018statistical,chen2020statistical}. The above theorem simply states this CLT does not hold under heavy-tailed gradient noise. Therefore, many of these procedures require further adaptation, if the gradient has undefined variance.
Finally, it is well-known that Polyak-Ruppert averaging achieves the Cram\'er-Rao lower bound~\citep{polyak1992acceleration,gadat2017optimal}, which is a lower bound on the variance of an unbiased estimator. However, it is not clear what this type of optimality means when the variance is not defined. These are important directions that require thorough investigations, and they will be studied elsewhere.

\section{Examples in the Presence of Heavy-tailed Noise}
\label{sec:examples}
In this section, we demonstrate how the stochastic approximation framework discussed in our paper covers several interesting examples, most notably linear regression and generalized linear models (GLMs), such that the heavy-tailed behavior naturally arise and the assumptions we proposed for Theorems~\ref{thm:rate-lp}, \ref{thm:stab-linear}, and \ref{thm:stab-nonlin} are all met.
\subsection{Ordinary Least Squares}
\par Let us first consider the following linear model,
\begin{equation}\label{eqn:m-lin-model}
    y =  \bm z\trsp \bm \beta_0  + \epsilon,
\end{equation}
where $\bm \beta_0 \in\reals^n$ is the true coefficients, $y \in \mathbb R$ is the response, the random vector $\bm z \in \mathbb R ^n$ denotes the covariates with a positive-definite second moment $0\prec\Exp[\bm z \bm z\trsp] < \infty$, and $\epsilon$ is a noise with zero conditional mean $\Exp[ \epsilon | \bm z ] = 0$.
In the classical setting, the noise $\epsilon$ is assumed to be Gaussian, whose variance is well defined. In this case, the population version of the maximum likelihood estimation (MLE) problem corresponds to minimizing $f(\bx) = \Exp[(y - \bm z\trsp \bx)^2]/2$ (where the expectation is taken over the $(y,\bm z)$ pair), or equivalently solving the following normal equations
\begin{equation}\label{eqn:lin-model-obj}
    \grad f(\bm x) \coloneqq \Exp\big[ \bm z \bm z \trsp\big] \bm x - \Exp[\bm z y] = 0.
\end{equation}
It can be easily verified that the true coefficients $\bm \beta_0$ is the unique zero of the
above equation, i.e. we have $\bx^* = \bm\beta_0$.

\par Now, suppose we are given access to a stream of i.i.d.\ drawn instances of the pair $(y, \bm z)$, denoted by $\{y_t, \bm z_t \}_{t\in \mathbb N^+}$. In large-scale settings, 
one generally runs the following stochastic approximation process, which is simply online SGD on the population MLE objective $f(\bx)$:
\begin{equation}\label{eqn:lin-model-SA-ver2}
    \bm x_{t} = \bm x_{t-1} - \gamma_t \left(\bm z_t \bm z_t\trsp \bm x_{t-1} - \bm z_t y_t\right).
\end{equation}
Manifestly, \eqref{eqn:lin-model-SA-ver2} is a special case of \eqref{eqn:SA},
where the gradient noise admitting the decomposition $\bm \xi_t = \bm \zeta_t+\bm m_t$,
for an i.i.d.\ component $\bm \zeta_t$ and a state-dependent component $\bm m_t$ (see \eqref{eqn:noise-decomp-2}),
\begin{equation}
    \begin{cases}\bm \zeta_t = \Exp[\bm z y] - \bm z_t y_t, \\\bm m_t = \left(\bm z_t \bm z_t\trsp - \Exp\left[ \bm z \bm z \trsp \right]\right)\bm x_{t-1}. \end{cases}
\end{equation}

In the presence of heavy-tailed noise, 
i.e., $\epsilon$ has possibly infinite variance,
the population MLE objective $f(\bx)$ may not be finite and
one should typically resort to methods from M-estimation and choose an appropriate loss function within robust statistics framework~\citep{huber2004robust,van2000asymptotic}. 
However, the SGD iterations \eqref{eqn:lin-model-SA-ver2} may still be employed to estimate the true coefficients $\bm\beta_0$ (potentially due to model misspecification), as we demonstrate below.

First, notice that 
the noise sequence can be decomposed in two parts, and the i.i.d. 
component $\{\bm \zeta_t\}_{t\in \mathbb N}$ exhibits the heavy-tailed behavior. 
Assume that this component has the highest defined moment order $1\le p<2$, i.e., $\Exp[|\bm \zeta_t|^p]<\infty$. Further,
the state dependent component $\bm m_t$ defines a martingale difference sequence, 
and the condition
\eqref{eqn:state-dep-var} is met since the covariates $\bm z$ 
have finite second moment, i.e.,
\begin{equation}
    \Exp\left[|\bm m_t|^2 \mid \bm x_{t-1}\right] \le C |\bm x_{t-1}|^2.
\end{equation}
Hence, Assumption~\ref{as:noise} is satisfied.
Next, assuming that the second moment of the covariates $\grad^2 f(\bx) = \Exp[\bm z \bm z\trsp]$ is $p$-PD, one can guarantee that Assumption~\ref{as:ppos} is satisfied.
Therefore, our convergence results can be invoked.
We emphasize that this assumption is always satisfied if $\Exp[\bm z \bm z\trsp]$ is diagonally dominant, but the condition is milder for $p>1$. %

\subsection{Generalized Linear Models}

In this section, we consider the problem of estimating the coefficients in generalized linear models (GLMs) in the presence of heavy-tailed noise.
GLMs play a crucial role in numerous statistics problems, and
provide a miscellaneous framework for many regression and classification tasks,
with many applications~\citep{mccullagh1989generalized,nelder1972generalized}.

For a response $y \in \mathbb R$ and random covariates $\bm z \in \mathbb R ^n$, the population version of an $\ell_2$-regularized MLE problem in the canonical GLM framework reads
\begin{equation}\label{objective:GLM}
\underset{\bx}{\text{minimize}} \ f(\bx)\coloneqq \Exp\left[\psi\left(\bx\trsp \bm z\right)- y\bx\trsp \bm z\right] + \frac{\lambda}{2} |\bx|^2\ \quad \ \text{ for }\ \quad \ \lambda>0.
\end{equation}
Here, $\psi : \reals \to \reals$ is referred to as the cumulant generating function (CGF) and assumed to be convex. Notable examples include $\psi(x)=x^2/2$ yielding linear regression, $\psi(x)=\log(1+e^{x})$ yielding logistic regression, and $\psi(x) = e^{x}$ yielding Poisson regression.
Gradient of the above objective \eqref{objective:GLM} is given by
\begin{equation}\label{eq:glm-grad}
    \grad f(\bm x) = \Exp\left[\bm z \psi'\big(\bm z\trsp\bm x\big)\right] - \Exp[\bm z y] + \lambda \bx.
\end{equation}

We define the unique solution of the population GLM problem as the unique zero of \eqref{eq:glm-grad}, which we denote by $\bx^*$. Note that we do not assume a model on data, allowing for model misspecification similar to~\citet{erdogdu2016scaled,erdogdu2019scalable}.
As in the previous section, 
we assume that the covariates have finite fourth moment and 
the response is contaminated with heavy-tailed noise with 
infinite variance. In this setting, the objective function is always defined,
even if the response has infinite variance. 

We are given access to a stream of i.i.d. drawn instances of the pair $(y, \bm z)$, denoted by $\{y_t, \bm z_t \}_{t\in \mathbb N^+}$, and we solve the above non-linear problem using the following stochastic process,
\begin{equation}\label{eqn:glin-model-SA}
    \bm x_{t} = \bm x_{t-1} - \gamma_t \left(\bm z_t \psi'\big(\bm z_t\trsp \bm x_{t-1}\big) - \bm z_t y_t + \lambda \bx_{t-1}\right),
\end{equation}
with gradient noise admitting the decomposition $\bm \xi_t = \bm \zeta_t + \bm m_t$ where
\begin{equation}
    \begin{cases}\bm \zeta_t = \Exp[\bm z y] - \bm z_t y_t, \\\bm m_t = \bm z_t \psi'\left(\bm z_t\trsp \bm x_{t-1}\right) - \Exp\left[\bm z_t \psi'\big(\bm z_t\trsp \bm x_{t-1}\big)\right]. \end{cases}
\end{equation}

In what follows, we verify our assumptions for a CGF
satisfying $|\psi'(x)|\le C(1+|x|)$ and $\psi''(x) \ge 0$ for all $x\in\reals$.
These assumptions can be easily verified for any convex CGF that grows at most linearly (e.g. $\psi(x) = \log(1+e^x)$).
$\bm\zeta_t$ are i.i.d.\ and contain the entire heavy-tailed part of the gradient noise. Assume that this component has the highest defined moment order $1\le p<2$, i.e., $\Exp[|\bm \zeta_t|^p]<\infty$.
Further observe that the state dependent component defines a martingale difference sequence and satisfies the condition
\eqref{eqn:state-dep-var} since the covariates $\bm z$ 
have finite fourth moment, and $|\psi'|$ grows at most linearly.
Therefore, Assumption~\ref{as:noise} is satisfied.

We note that the Hessian of the objective $f$ is given as
\begin{equation}\label{eq:glm-hess}
    \grad^2 f(\bm x) = \Exp\left[\bm z\bm z\trsp \psi''\big(\bm z\trsp\bm x\big)\right] + \lambda \bf I.
\end{equation}
Since $\psi'' (x)\ge 0$, $\grad^2 f(\bm x)$ is clearly PD for all $\lambda >0$. For sufficiently large $\lambda$, this matrix can also be made diagonally dominant,
which implies that it is $p$-PD for any $p\ge 1$, further implying Assumption~\ref{as:ppos}.
Therefore, for an appropriate step-size sequence, our convergence results on the SGD can be applied to this framework.

\section{Conclusion}
In this paper, we considered SGD subject to state-dependent and heavy-tailed noise with a potentially infinite variance when the objective belongs to a class of strongly convex functions. We provided a convergence rate for the distance to the optimizer in $L^p$ under appropriate assumptions. Furthermore, we provided a generalized central limit theorem that shows that the averaged iterates converge to a multivariate $\alpha$-stable distribution. We also discussed the implications of our results to applications such as linear regression and generalized linear models subject to heavy-tailed input data. Finally, 
while we leave it for a future study, we emphasize the importance of adapting existing statistical inference techniques that rely on the averaged SGD iterates in the presence of heavy-tailed gradient noise which arises naturally in modern statistical learning applications.

\section*{Acknowledgements}
MAE is partially funded by CIFAR AI Chairs program, and CIFAR AI Catalyst grant. 
MG’s research is supported in part by
the grants NSF DMS-1723085 and NSF CCF-1814888.
LZ is grateful to the support from a Simons Foundation Collaboration Grant.

\newpage
\bibliographystyle{abbrvnat}
{\small \bibliography{notes}}

\begin{thebibliography}{52}
\providecommand{\natexlab}[1]{#1}
\providecommand{\url}[1]{\texttt{#1}}
\expandafter\ifx\csname urlstyle\endcsname\relax
  \providecommand{\doi}[1]{doi: #1}\else
  \providecommand{\doi}{doi: \begingroup \urlstyle{rm}\Url}\fi

\bibitem[Anantharam and Borkar(2012)]{anantharam2012stochastic}
V.~Anantharam and V.~S. Borkar.
\newblock Stochastic approximation with long range dependent and heavy tailed
  noise.
\newblock \emph{Queueing Systems}, 71\penalty0 (1-2):\penalty0 221--242, 2012.

\bibitem[Anastasiou et~al.(2019)Anastasiou, Balasubramanian, and
  Erdogdu]{anastasiou2019normal}
A.~Anastasiou, K.~Balasubramanian, and M.~A. Erdogdu.
\newblock Normal approximation for stochastic gradient descent via
  non-asymptotic rates of martingale {CLT}.
\newblock In \emph{Conference on Learning Theory}, pages 115--137, 2019.

\bibitem[Blum(1954)]{blum1954approximation}
J.~R. Blum.
\newblock Approximation methods which converge with probability one.
\newblock \emph{The Annals of Mathematical Statistics}, 25\penalty0
  (2):\penalty0 382--386, 1954.

\bibitem[Buraczewski et~al.(2012)Buraczewski, Damek, and
  Mirek]{buraczewski2012asymptotics}
D.~Buraczewski, E.~Damek, and M.~Mirek.
\newblock Asymptotics of stationary solutions of multivariate stochastic
  recursions with heavy tailed inputs and related limit theorems.
\newblock \emph{Stochastic Processes and their Applications}, 122\penalty0
  (1):\penalty0 42--67, 2012.

\bibitem[Buraczewski et~al.(2016)Buraczewski, Damek, and
  Mikosch]{buraczewski2016stochastic}
D.~Buraczewski, E.~Damek, and T.~Mikosch.
\newblock \emph{Stochastic Models with Power-Law Tails}.
\newblock Springer, 2016.

\bibitem[Chen et~al.(2020)Chen, Lee, Tong, and Zhang]{chen2020statistical}
X.~Chen, J.~D. Lee, X.~T. Tong, and Y.~Zhang.
\newblock Statistical inference for model parameters in stochastic gradient
  descent.
\newblock \emph{The Annals of Statistics}, 48\penalty0 (1):\penalty0 251--273,
  2020.

\bibitem[Cherapanamjeri et~al.(2020)Cherapanamjeri, Tripuraneni, Bartlett, and
  Jordan]{cherapanamjeri2020optimal}
Y.~Cherapanamjeri, N.~Tripuraneni, P.~L. Bartlett, and M.~I. Jordan.
\newblock Optimal mean estimation without a variance.
\newblock \emph{arXiv preprint arXiv:2011.12433}, 2020.

\bibitem[Chung(1954)]{chung1954stochastic}
K.~L. Chung.
\newblock On a stochastic approximation method.
\newblock \emph{The Annals of Mathematical Statistics}, 25\penalty0
  (3):\penalty0 463--483, 1954.

\bibitem[Davis et~al.(2019)Davis, Drusvyatskiy, Xiao, and Zhang]{davis2019low}
D.~Davis, D.~Drusvyatskiy, L.~Xiao, and J.~Zhang.
\newblock From low probability to high confidence in stochastic convex
  optimization.
\newblock \emph{arXiv preprint arXiv:1907.13307}, 2019.

\bibitem[Duchi and Ruan(2016)]{duchi2016local}
J.~Duchi and F.~Ruan.
\newblock Asymptotic optimality in stochastic optimization.
\newblock \emph{arXiv:1612.05612}, 2016.

\bibitem[Erdogdu et~al.(2016)Erdogdu, Bayati, and Dicker]{erdogdu2016scaled}
M.~A. Erdogdu, M.~Bayati, and L.~H. Dicker.
\newblock Scaled least squares estimator for glms in large-scale problems.
\newblock In \emph{Proceedings of the 30th International Conference on Neural
  Information Processing Systems}, pages 3332--3340, 2016.

\bibitem[Erdogdu et~al.(2019)Erdogdu, Bayati, and Dicker]{erdogdu2019scalable}
M.~A. Erdogdu, M.~Bayati, and L.~H. Dicker.
\newblock Scalable approximations for generalized linear problems.
\newblock \emph{The Journal of Machine Learning Research}, 20\penalty0
  (1):\penalty0 231--275, 2019.

\bibitem[Fabian(1967)]{fabian1967stochastic}
V.~Fabian.
\newblock Stochastic approximation of minima with improved asymptotic speed.
\newblock \emph{The Annals of Mathematical Statistics}, 38\penalty0
  (1):\penalty0 191--200, 1967.

\bibitem[Fabian(1968)]{fabian1968asymptotic}
V.~Fabian.
\newblock On asymptotic normality in stochastic approximation.
\newblock \emph{The Annals of Mathematical Statistics}, 39\penalty0
  (4):\penalty0 1327--1332, 1968.

\bibitem[Fang et~al.(2018)Fang, Xu, and Yang]{fang2018online}
Y.~Fang, J.~Xu, and L.~Yang.
\newblock Online bootstrap confidence intervals for the stochastic gradient
  descent estimator.
\newblock \emph{The Journal of Machine Learning Research}, 19\penalty0
  (1):\penalty0 3053--3073, 2018.

\bibitem[Farsad et~al.(2015)Farsad, Guo, Chae, and Eckford]{Farsad2015}
N.~Farsad, W.~Guo, C.~B. Chae, and A.~Eckford.
\newblock {Stable distributions as noise models for molecular communication}.
\newblock \emph{2015 IEEE Global Communications Conference, GLOBECOM 2015},
  2015.

\bibitem[Feldheim(1937)]{Feldheim}
E.~Feldheim.
\newblock \emph{\'{E}tude de la stabilit\'{e} des lois de probabilit\'{e}}.
\newblock PhD thesis, Facult\'{e} des Sciences de Paris, Paris, 1937.

\bibitem[Feller(1971)]{Feller}
W.~Feller.
\newblock \emph{An Introduction to Probability Theory and Its Applications}.
\newblock Wiley, New York, 2nd edition, 1971.

\bibitem[Fiche et~al.(2013)Fiche, Cexus, Martin, and Khenchaf]{Fiche2013}
A.~Fiche, J.~C. Cexus, A.~Martin, and A.~Khenchaf.
\newblock {Features modeling with an $\alpha$-stable distribution: Application
  to pattern recognition based on continuous belief functions}.
\newblock \emph{Information Fusion}, 14\penalty0 (4):\penalty0 504--520, 2013.

\bibitem[Gadat and Panloup(2017)]{gadat2017optimal}
S.~Gadat and F.~Panloup.
\newblock Optimal non-asymptotic bound of the {R}uppert-{P}olyak averaging
  without strong convexity.
\newblock \emph{arXiv preprint arXiv:1709.03342}, 2017.

\bibitem[Geluk and de~Hann(2000)]{Geluk}
J.~L. Geluk and L.~de~Hann.
\newblock Stable probability distributions and their domains of attraction: A
  direct approach.
\newblock \emph{Probability and Mathematical Statistics}, 20:\penalty0
  169--188, 2000.

\bibitem[Gnedenko and Kolmogorov(1954)]{gnedenko1968limit}
B.~V. Gnedenko and A.~Kolmogorov.
\newblock \emph{{Limit Distributions for Sums of Independent Random
  Variables}}.
\newblock Addison-Wesley, Cambridge, MA, 1954.
\newblock Translated by Kai Lai Chung.

\bibitem[Goodsell and Hanson(1976)]{goodsell1976almost}
C.~Goodsell and D.~Hanson.
\newblock Almost sure convergence for the {R}obbins-{M}onro process.
\newblock \emph{The Annals of Probability}, 4\penalty0 (6):\penalty0 890--901,
  1976.

\bibitem[Gorbunov et~al.(2020)Gorbunov, Danilova, and
  Gasnikov]{gorbunov2020stochastic}
E.~Gorbunov, M.~Danilova, and A.~Gasnikov.
\newblock Stochastic optimization with heavy-tailed noise via accelerated
  gradient clipping.
\newblock In \emph{Advances in Neural Information Processing Systems},
  volume~33, 2020.

\bibitem[G\"{u}rb\"{u}zbalaban and Hu(2020)]{gurbuzbalaban2020fractional}
M.~G\"{u}rb\"{u}zbalaban and Y.~Hu.
\newblock Fractional moment-preserving initialization schemes for training
  fully-connected neural networks.
\newblock \emph{arXiv preprint arXiv:2005.11878}, 2020.

\bibitem[G\"urb\"uzbalaban et~al.(2020)G\"urb\"uzbalaban, \c{S}im\c{s}ekli, and
  Zhu]{gurbuzbalaban2020heavy}
M.~G\"urb\"uzbalaban, U.~\c{S}im\c{s}ekli, and L.~Zhu.
\newblock The heavy-tail phenomenon in {SGD}.
\newblock \emph{arXiv preprint arXiv:2006.04740}, 2020.

\bibitem[Hodgkinson and Mahoney(2020)]{hodgkinson2020multiplicative}
L.~Hodgkinson and M.~W. Mahoney.
\newblock Multiplicative noise and heavy tails in stochastic optimization.
\newblock \emph{arXiv preprint arXiv:2006.06293}, 2020.

\bibitem[Huber(2004)]{huber2004robust}
P.~J. Huber.
\newblock \emph{Robust Statistics}, volume 523.
\newblock John Wiley \& Sons, 2004.

\bibitem[Krasulina(1969)]{krasulina1969stochastic}
T.~P. Krasulina.
\newblock On stochastic approximation processes with infinite variance.
\newblock \emph{Theory of Probability \& Its Applications}, 14\penalty0
  (3):\penalty0 522--526, 1969.

\bibitem[Kushner and Yin(2003)]{kushner2003stochastic}
H.~Kushner and G.~G. Yin.
\newblock \emph{Stochastic Approximation and Recursive Algorithms and
  Applications}, volume~35.
\newblock Springer Science \& Business Media, 2003.

\bibitem[L{\'e}vy(1937)]{paul1937theorie}
P.~L{\'e}vy.
\newblock Th{\'e}orie de l'addition des variables al{\'e}atoires.
\newblock \emph{Gauthiers-Villars, Paris}, 1937.

\bibitem[Li(1994)]{li1994almost}
G.~Li.
\newblock Almost sure convergence of stochastic approximation procedures.
\newblock \emph{Statistica Sinica}, 4\penalty0 (1):\penalty0 361--372, 1994.

\bibitem[McCullagh and Nelder(1989)]{mccullagh1989generalized}
P.~McCullagh and J.~A. Nelder.
\newblock \emph{{Generalized Linear Models}}.
\newblock Chapman and Hall, 2nd edition, 1989.

\bibitem[Mirek(2011)]{mirek2011heavy}
M.~Mirek.
\newblock Heavy tail phenomenon and convergence to stable laws for iterated
  {L}ipschitz maps.
\newblock \emph{Probability Theory and Related Fields}, 151\penalty0
  (3):\penalty0 705--734, 2011.

\bibitem[Nazin et~al.(2019)Nazin, Nemirovsky, Tsybakov, and
  Juditsky]{nazin2019algorithms}
A.~V. Nazin, A.~S. Nemirovsky, A.~B. Tsybakov, and A.~B. Juditsky.
\newblock Algorithms of robust stochastic optimization based on mirror descent
  method.
\newblock \emph{Automation and Remote Control}, 80\penalty0 (9):\penalty0
  1607--1627, 2019.

\bibitem[Nelder and Wedderburn(1972)]{nelder1972generalized}
J.~A. Nelder and R.~W. Wedderburn.
\newblock Generalized linear models.
\newblock \emph{Journal of the Royal Statistical Society: Series A (General)},
  135\penalty0 (3):\penalty0 370--384, 1972.

\bibitem[Neveu(1975)]{neveu1975discrete}
J.~Neveu.
\newblock \emph{Discrete-Parameter Martingales}, volume~10.
\newblock North-Holland Amsterdam, 1975.

\bibitem[Polyak and Juditsky(1992)]{polyak1992acceleration}
B.~T. Polyak and A.~B. Juditsky.
\newblock Acceleration of stochastic approximation by averaging.
\newblock \emph{SIAM Journal on Control and Optimization}, 30\penalty0
  (4):\penalty0 838--855, 1992.

\bibitem[Robbins and Monro(1951)]{robbins1951stochastic}
H.~Robbins and S.~Monro.
\newblock A stochastic approximation method.
\newblock \emph{The Annals of Mathematical Statistics}, 22\penalty0
  (3):\penalty0 400--407, 1951.

\bibitem[Ruppert(1988)]{ruppert1988efficient}
D.~Ruppert.
\newblock Efficient estimations from a slowly convergent {R}obbins-{M}onro
  process.
\newblock Technical report, Cornell University Operations Research and
  Industrial Engineering, 1988.

\bibitem[Sacks(1958)]{sacks1958asymptotic}
J.~Sacks.
\newblock Asymptotic distribution of stochastic approximation procedures.
\newblock \emph{The Annals of Mathematical Statistics}, 29\penalty0
  (2):\penalty0 373--405, 1958.

\bibitem[Samorodnitsky and Taqqu(1994)]{ST1994}
G.~Samorodnitsky and M.~S. Taqqu.
\newblock \emph{Stable Non-Gaussian Random Processes: Stochastic Models with
  Infinite Variance}.
\newblock Chapman \& Hall, New York, 1994.

\bibitem[Sarafrazi and Yazdi(2019)]{Sarafrazi2019}
K.~Sarafrazi and M.~Yazdi.
\newblock {Skewed alpha-stable distribution for natural texture modeling and
  segmentation in contourlet domain}.
\newblock \emph{Eurasip Journal on Image and Video Processing}, 2019\penalty0
  (1):\penalty0 1--12, 2019.

\bibitem[Shapiro(1989)]{shapiro1989asymptotic}
A.~Shapiro.
\newblock Asymptotic properties of statistical estimators in stochastic
  programming.
\newblock \emph{The Annals of Statistics}, 17\penalty0 (2):\penalty0 841--858,
  1989.

\bibitem[{\c{S}}im{\c{s}}ekli et~al.(2019){\c{S}}im{\c{s}}ekli,
  G{\"u}rb{\"u}zbalaban, Nguyen, Richard, and Sagun]{csimcsekli2019heavy}
U.~{\c{S}}im{\c{s}}ekli, M.~G{\"u}rb{\"u}zbalaban, T.~H. Nguyen, G.~Richard,
  and L.~Sagun.
\newblock On the heavy-tailed theory of stochastic gradient descent for deep
  neural networks.
\newblock \emph{arXiv preprint arXiv:1912.00018}, 2019.

\bibitem[{\c{S}im\c{s}ekli} et~al.(2019){\c{S}im\c{s}ekli}, Sagun, and
  G\"{u}rb\"{u}zbalaban]{simsekli19a}
U.~{\c{S}im\c{s}ekli}, L.~Sagun, and M.~G\"{u}rb\"{u}zbalaban.
\newblock A tail-index analysis of stochastic gradient noise in deep neural
  networks.
\newblock In \emph{International Conference on Machine Learning}, pages
  5827--5837, 2019.

\bibitem[Su and Zhu(2018)]{su2018statistical}
W.~Su and Y.~Zhu.
\newblock Statistical inference for online learning and stochastic
  approximation via hierarchical incremental gradient descent.
\newblock \emph{arXiv preprint arXiv:1802.04876}, 2018.

\bibitem[Toulis and Airoldi(2017)]{toulis2017asymptotic}
P.~Toulis and E.~M. Airoldi.
\newblock Asymptotic and finite-sample properties of estimators based on
  stochastic gradients.
\newblock \emph{The Annals of Statistics}, 45\penalty0 (4):\penalty0
  1694--1727, 2017.

\bibitem[Tripuraneni et~al.(2018)Tripuraneni, Flammarion, Bach, and
  Jordan]{tripuraneni2018averaging}
N.~Tripuraneni, N.~Flammarion, F.~Bach, and M.~I. Jordan.
\newblock Averaging stochastic gradient descent on {R}iemannian manifolds.
\newblock In \emph{Proceedings of the 31st Conference On Learning Theory},
  2018.

\bibitem[Van~der Vaart(2000)]{van2000asymptotic}
A.~W. Van~der Vaart.
\newblock \emph{Asymptotic Statistics}, volume~3.
\newblock Cambridge University Press, 2000.

\bibitem[Yu et~al.(2020)Yu, Balasubramanian, Volgushev, and
  Erdogdu]{yu2020analysis}
L.~Yu, K.~Balasubramanian, S.~Volgushev, and M.~A. Erdogdu.
\newblock An analysis of constant step size sgd in the non-convex regime:
  Asymptotic normality and bias.
\newblock \emph{arXiv preprint arXiv:2006.07904}, 2020.

\bibitem[Zhang et~al.(2019)Zhang, Karimireddy, Veit, Kim, Reddi, Kumar, and
  Sra]{zhang2019adam}
J.~Zhang, S.~P. Karimireddy, A.~Veit, S.~Kim, S.~J. Reddi, S.~Kumar, and
  S.~Sra.
\newblock Why are adaptive methods good for attention models?
\newblock \emph{arXiv preprint arXiv:1912.03194}, 2019.

\end{thebibliography}
\newpage
\appendix

\section{Lemmas and Discussions}
\label{sec:discussion}

\subsection{Key Lemmas}
\par In this subsection, we present some key lemmas used in the proof of our main theorems, which are helpful when considering stochastic problems with \emph{infinite} variance.

\par The concept of \emph{uncorrelatedness} has long been used by probabilists as a trick when computing and estimating variance. For example, consider a sequence of uncorrelated random vectors $\{\bm X_t\}_{t\in\mathbb N^+}$ (e.g. square-integrable martingale difference). Then
\begin{equation}\label{eqn:expansion-l2-uncorr}
    \Exp\left[|\bm X_1 + \ldots + \bm X_t|^2 \right] = \Exp\left[|\bm X_1|^2  \right] + \ldots + \Exp\left[|\bm X_t|^2  \right].
\end{equation}
Indeed, this type of expansion is used in \cite{polyak1992acceleration} to show $L^2$ convergence in the normality analysis of stochastic approximation problems.

\par However, correlatedness is \emph{only} defined when random elements have \emph{finite} variance. The following lemma provides an infinite-variance version of expansion \eqref{eqn:expansion-l2-uncorr}, stating that the $p$-th moment ($p < 2$) of a martingale without square-integrability assumption can also be bounded \emph{simpliciter} by the sum of the $p$-th moments of its differences, at the cost of a multiplicative constant that may depend only on $p$ and the dimension $n$. It is a generalization of the recent study \citet[Lemma~4.2]{cherapanamjeri2020optimal}.

\begin{lemma}\thlabel{lem:p-expand} Suppose $p \in [0, 1]$ and let $\{\bm S_t \}_{t\in\mathbb N}$ be an $n$-dimensional martingale adapted to the filtration $\{ \mathcal F_t \}_{t\in\mathbb N}$, with $\Exp[ | \bm S_t |^{1+p} ] < \infty$ for every $t$ and $\bm S_0 = 0$. Let $\bm X_i = \bm S_i - \bm S_{i-1}$. Then
\begin{align}
\Exp\left[ \left| \bm S_t \right|^{1+p}\right] 
\le 2^{1-p}n^{1-\frac{1+p}{2}} \sum_{i=1}^t \Exp\left[ \left| \bm X_i \right|^{1+p}\right].
\end{align}
\end{lemma}

\par Next, we present a Taylor-expansion-type inequality for the function $\|\bm x \|_p^p$. Recall that we have defined the signed power of a vector in \eqref{eqn:signed-power}.

\begin{lemma}\thlabel{lem:vecexpandp} Let $p\in[1,2]$. For any $\bm x, \bm y \in \mathbb R^n$, $\|\bm x+\bm y\|^p_p\le\|\bm x\|^p_p + 4\|\bm y\|^p_p + p \bm y \trsp \bm x^{\langle p-1 \rangle}$.
\end{lemma}

This inequality traces back to \cite{krasulina1969stochastic}, where the one-dimensional version $|x+y|^p \le |x|^p + C|y|^p+pyx^{p-1}\sg(x)$ is used\footnote{The paper \cite{krasulina1969stochastic} contains a minor error in ignoring the signum function $\sg(x)$ in this inequality. Our proof of \thref{thm:rate-lp} can be thought of its correction as well as extension.} to derive an $L^p$ rate of convergence for the one-dimensional stochastic approximation process with step-size $1/t$. In our current study, this lemma is used not only to derive $L^p$ rate of convergence for general infinite-variance process in $\mathbb R^n$ with variable step-size scheme (\thref{thm:rate-lp}), but also in the proof of the equivalent definitions of $p$-PD (\thref{thm:equivalent-p+}).

\par Finally, we quote \citet[Lemma~4.2]{fabian1967stochastic}, which we shall use to calculate the exact convergence rate (see also \cite{chung1954stochastic}).

\begin{lemma}[\cite{fabian1967stochastic}, Lemma~4.2]\thlabel{lem:chung-fabian} Let $\{b_t\}_{t\in\mathbb N}, A, B, \alpha, \beta$ be real numbers such that $0
<\alpha < 1$, $A>0$ and suppose the recursion
\begin{equation}
    b_{t+1} = b_t(1-At^{-\alpha}) + Bt^{-\alpha-\beta}
\end{equation}
holds. Then, $b_t = \Theta(t^{-\beta})$.
\end{lemma}

\subsection{Discussions on $p$-Positive Definiteness and Uniform $p$-Positive Definiteness}
\label{sec:p+}
\par Let us now focus on $p$-PD and uniform $p$-PD assumptions which we defined back in \thref{def:p+} and \thref{def:unfmp+}. Our next theorem provides several equivalent characterizations of $p$-PD. 

\begin{theorem}[Equivalent definitions of $p$-PD]\thlabel{thm:equivalent-p+} Let $\bf Q$ be a symmetric matrix. The following are equivalent when $p\in[1,2]$.
\begin{itemize}
    \item[i)] There exist $\delta, L > 0$, such that $\| \bf I - t \bf Q \|_p^p \le 1 - Lt$ for all $t\in[0,\delta)$.
    \item[ii)] There exists $\lambda > 0$ such that for all $\bm v \in \mathbb R^n$, $\bm v \trsp \bf Q \bm v^{\langle p - 1 \rangle} \ge \lambda \|\bm v\|_p^p$. 
    \item[iii)] For all $\bm v \in \sphere_p$, $\bm v \trsp \bf Q \bm v^{\langle p - 1 \rangle} > 0$.
    \item[iv)] For all $\bm v \in\sphere_p$, there exists $t_0 > 0$ such that $\|\bm v - t_0 \bf Q \bm v \|_p < 1$.
\end{itemize}
\end{theorem}

Next, we provide several equivalent characterizations of uniform $p$-PD. 

\begin{theorem}[Equivalent definitions of uniform $p$-PD]\thlabel{thm:equivalent-unfmp+} Let $\mathcal{M}$ be a bounded set of symmetric matrices. The following are equivalent when $p\in[1,2]$.
\begin{itemize}
    \item[i)] There exist $\delta, L > 0$, such that $\| \bf I - t \bf Q \|_p^p \le 1 - Lt$ for all $t\in[0,\delta)$ and $\bf Q \in \mathcal M$.
    \item[ii)] There exists $\lambda > 0$ such that for all $\bm v \in \mathbb R^n$ and $\bf Q \in \mathcal M$, $\bm v \trsp \bf Q \bm v^{\langle p - 1 \rangle} \ge \lambda \|\bm v\|_p^p$.
    \item[iii)] For all $\bm v \in \sphere_p$ and $\bf Q \in \overline{\mathcal M}$, $\bm v \trsp \bf Q \bm v^{\langle p - 1 \rangle} > 0$.
    \item[iv)] For all $\bm v \in\sphere_p$ and $\bf Q \in \overline{\mathcal M}$, there exists $t_0 > 0$ such that $\|\bm v - t_0 \bf Q \bm v \|_p < 1$.
\end{itemize}
\end{theorem}

We notice that some mild assumptions can indeed imply $p$-PD. For example, we will show that diagonal dominance implies $p$-PD. Recall that a symmetric matrix $\bf Q = (q_{ij})_{n\times n}$ is called diagonally dominant (with non-negative diagonal) if for every $i\in[n]$,
\begin{equation}
    q_{ii} - \sum_{j\in[n] \setminus \{ i \}} |q_{ij}| > 0.
\end{equation}
Further, we say that a non-empty set $\mathcal M$ of symmetric matrices is \emph{uniformly diagonally dominant} (with non-negative diagonal) if
\begin{equation}
    \inf_{(q_{ij})_{n\times n} \in \mathcal M  }\min_{i\in[n] }\left( q_{ii} - \sum_{j\in[n] \setminus \{ i \}} |q_{ij}| \right) > 0.
\end{equation}
We have the following observations which we shall prove in Appendix~\ref{sec:proofs}. First, we observe
that the uniform $p$-PD assumption is weaker than
the notion of uniform diagonally dominance (with non-negative diagonal).

\begin{proposition}\thlabel{thm:diagdom->ppos}
A uniformly diagonally dominant (with non-negative diagonal) set of symmetric matrices is uniformly $p$-PD for every $p\in [1,2]$.
\end{proposition}

Next, we notice that the result in Proposition~\ref{thm:diagdom->ppos} is tight
for $p=1$.

\begin{proposition}\thlabel{rmk:1pd=udd}
Uniform $1$-PD is equivalent to uniform diagonal dominance (with non-negative diagonal).
\end{proposition}

Finally, we observe that the notion of uniform $2$-PD 
is weaker than uniform $p$-PD for any $p\in[1,2]$.

\begin{proposition}\thlabel{rmk:ppd->2pd}
Let $p\in[1,2]$. Uniform $p$-PD implies uniform 2-PD.
\end{proposition}

\section{Omitted Proofs}\label{sec:proofs}
In this appendix, we first prove the lemmas, theorems, and propositions in Section~\ref{sec:discussion}, then prove the theorems in Sections~\ref{sec:convergence} and \ref{sec:averaging}. Throughout this appendix, we denote by $\bm \delta_t$ the error of the approximation $\bx_t - \bx^*$, and by $\overline{\bm \delta}_t$ the averaged error $(\bm \delta_0 + \ldots + \bm \delta_{t-1})/t$. The gradient $\grad f (\bx)$ and the Hessian $\grad^2 f(\bx)$ will be written as $\bm R(\bx)$ and $\grad \bm R(\bx)$ respectively, not only for notational simplicity, but also to stress the fact that our results can be applied to any instance of stochastic approximation \eqref{eqn:SA} including SGD.

\begin{proof-of}[\thref{lem:p-expand}]We first prove the $n=1$ case. Suppose $\{S_t\}$ is a one-dimensional martingale and $X_i = S_i - S_{i-1}$. Notice that the function $g(x) = |x|^{1+p}$ satisfies the inequality (see e.g. \citet[Lemma~A.3]{cherapanamjeri2020optimal}):
\begin{equation}
    |g'(x) - g'(y)| \le 2^{1-p} g'(|x-y|),
\end{equation}
where the weak derivative $g'(x) = \sg (x)$ is used in the inequality above in the case of $p=0$, where $$\sg(x):=\begin{cases} 1 & \mbox{if} \quad x>0, \\
                                 -1 & \mbox{if} \quad x<0, \\
                                  0 & \mbox{if} \quad x=0.
                                  \end{cases}$$
Furthermore, by $\Exp[X_i g'(S_{i-1}) \mid \mathcal F_{i-1}] = g'(S_{i-1}) \Exp[ X_i  \mid \mathcal F_{i-1} ] = 0$, we have
\begin{align}
    \Exp[g(S_t)] &= \sum_{i=1}^t \Exp \left[ \int_{S_{i-1}}^{S_i} g'(x) \d x \right]\nonumber \\
    &= \sum_{i=1}^t \Exp \left[ X_i g'(S_{i-1}) + \int_{S_{i-1}}^{S_i} \left[g'(x) - g'(S_{i-1})\right] \d x \right]\nonumber \\
    &= \sum_{i=1}^t \Exp \left[ \int_{S_{i-1}}^{S_i} \left[g'(x) - g'(S_{i-1})\right] \d x \right] \nonumber\\
    &= \sum_{i=1}^t \Exp \left[ \int_{0}^{X_i} \left[g'(S_{i-1} + \tau) - g'(S_{i-1})\right] \d \tau \right]\nonumber \\
    &= \sum_{i=1}^t \Exp \left[ \int_{0}^{|X_i|} \left|g'(S_{i-1} + \sg(X_i)\tau) - g'(S_{i-1})\right| \d \tau \right]\nonumber \\
    &\le 2^{1-p} \sum_{i=1}^t \Exp \left[ \int_{0}^{|X_i|} g'(\tau) \d \tau \right] \nonumber\\
    & =   2^{1-p}  \sum_{i=1}^t \Exp[g(|X_i|)].\label{ineq:n:1}
\end{align}

\par Next, for the higher dimension $n>1$, we denote by $S^j_i$ (resp. $X^j_i$) the $j$-th entry of the vector $\bm S_i$ (resp. $\bm X_i$). We can apply the inequality \eqref{ineq:n:1} obtained above to $S_t^j$ by taking a $(1+p)$-norm,
\begin{align}
    \Exp \left[ \left\| \bm S_t \right\|^{1+p}_{1+p} \right] &= \sum_{j=1}^n \Exp\left[ \left| S_t^j \right|^{1+p} \right]  \nonumber\\
    &\le \sum_{j=1}^n 2^{1-p} \sum_{i=1}^t \Exp \left[ \left|X_i^j\right|^{1+p} \right] \nonumber\\
    & = 2^{1-p} \sum_{i=1}^t  \sum_{j=1}^n \Exp \left[ \left|X_i^j\right|^{1+p} \right] \nonumber\\
    &= 2^{1-p}  \sum_{i=1}^t \Exp \left[\left\| \bm X_i \right\|^{1+p}_{1+p} \right].
\end{align} 
Finally, the inequalities
\begin{equation}
    |\bm x| \le \| \bm x \|_{1+p} \le n^{\frac{1}{1+p} - \frac{1}{2}} |\bm x|
\end{equation}
give our desired result:
\begin{equation}
  \Exp \left[ | \bm S_t |^{1+p} \right] \le 2^{1-p}n^{1-\frac{1+p}{2}} \sum_{i=1}^t \Exp \left[| \bm X_i |^{1+p} \right].
\end{equation}
The proof is complete.
\end{proof-of}

\begin{proof-of}[\thref{lem:vecexpandp}]By the inequality that $|1+a|^p \le 1 + ap + 4|a|^p$ for any $p \in [1,2]$ and $a \in \mathbb R$, we have that for any $p\in [1,2]$
and $x, y \in \mathbb R$,
\begin{align}\label{one:dim:x:y}
    |x+y|^p \le |x|^p + py|x|^{p-1} \sg(x) + 4|y|^p.
\end{align}
Next, for any $\bm x = (x^1, \ldots, x^n)\trsp, \bm y = (y^1, \ldots, y^n)\trsp \in \mathbb R^n$, by taking the $p$-norm and applying the inequality \eqref{one:dim:x:y}, we obtain
\begin{align}
    \|\bm x+\bm y\|^p_p &= \sum_{i=1}^n \left|x^i + y^i \right|^p \nonumber\\
    &\le  \sum_{i=1}^n \left( \left|x^i\right|^p + py^i\left|x^i\right|^{p-1} \sg(x^i) + 4\left|y^i\right|^p \right)\nonumber\\
    &= \|\bm x\|^p_p + 4\|\bm y\|^p_p + p \sum_{i=1}^n y^i\left|x^i\right|^{p-1} \sg(x^i)\label{apply:ineq} \\
    &= \|\bm x\|^p_p + 4\|\bm y\|^p_p + p \bm y \trsp \bm x^{\langle p-1 \rangle},\label{poweredineq}
\end{align}
which completes the proof.
\end{proof-of}

Since \thref{thm:equivalent-p+} is just a special case of \thref{thm:equivalent-unfmp+}, we will only prove the latter. Before we proceed, let us first state a useful technical lemma.

\begin{lemma}\thlabel{lem:convex:varphi}
Let $\bm u, \bm v \in \mathbb R^n$ and consider the function $\varphi(t) = \| \bm u + t \bm v \|_p^p = \sum_{i=1}^n | u^i + v^i t |^p$. The function $\varphi$ is convex and has the following derivative (when $1<p\le 2$) or subderivative (when $p=1$):
\begin{equation}
   \varphi'(t) = \sum_{i=1}^n p\left|u^i + v^i t\right|^{p-1} \sg\left(u^i + v^i t\right) v^i   = p\bm v\trsp (\bm u + t \bm v)^{\langle p -1 \rangle}.
\end{equation}
\end{lemma}
The proof of Lemma~\ref{lem:convex:varphi} is straightforward and is hence omitted here.

\par Now we are ready to prove \thref{thm:equivalent-unfmp+}.

\begin{proof-of}[\thref{thm:equivalent-unfmp+}]We shall show that i)$\implies$iv)$\implies$iii)$\implies$ii)$\implies$i).
\begin{itemize}[leftmargin=1in]
    \item[i)$\implies$iv)] Take a sequence $\{\bf Q_1, \bf Q_2, \ldots \} \subseteq \mathcal M$ such that $\lim_{m\to\infty} \bf Q_m = \bf Q$. iv) follows from $\|\bf I - (\delta/2) \bf Q_m\|_p^p \le 1 - L\delta/2$.
    \item[iv)$\implies$iii)] For all $\bm v \in \sphere_p$ and $\bf Q \in \overline{\mathcal M}$, consider the function $\varphi(t) = \|\bm v - t\bf Q \bm v  \|_p^p$. According to \thref{lem:convex:varphi}, $\varphi(t)$ is convex. Furthermore, $\varphi(t_0) < 1 = \varphi(0)$. Hence it follows that $\varphi'(0) < 0$; that is, $\bm v \trsp \bf Q \bm v^{\langle p - 1 \rangle} > 0$.
    \item[iii)$\implies$ii)] Since the function $(\bm v, \bf Q) \mapsto\bm v \trsp \bf Q \bm v^{\langle p - 1 \rangle}$ is continuous, it maps the compact set $\sphere_p \times \overline{\mathcal M}$ to a compact set. Hence there exists some $\lambda > 0$ such that for all $\bm v \in \sphere_p$ and $\bf Q \in \overline{\mathcal M}$, $\bm v \trsp \bf Q \bm v^{\langle p - 1 \rangle} \ge \lambda$. Now, for every $\bm u \in \mathbb R^n \setminus \{0\}$, by setting $\bm v = \bm u / \|\bm u\|_p$, we get $\bm u \trsp \bf Q \bm u^{\langle p - 1 \rangle} \ge \lambda \|\bm u\|_p^p$.
    \item[ii)$\implies$i)] For arbitrary $\bm v \in \mathbb R^n$ and $\bf Q \in \mathcal{M}$, by \thref{lem:vecexpandp} we have $\|(\bf I - t \bf Q)\bm v\|_p^p = \| \bm v - t \bf Q \bm v \|_p^p \le \|\bm v\|_p^p + 4t^p \| \bf Q \bm v \|_p^p - pt (\bm v\trsp \bf Q \bm v^{\langle p-1\rangle}) \le \|\bm v\|_p^p + 4t^p \|\bf Q\|_p^p \|\bm v\|_p^p - pt\lambda\|\bm v\|_p^p$. This implies i).
\end{itemize}
The proof is complete.
\end{proof-of}

\begin{proof-of}[\thref{thm:diagdom->ppos}]Let $\bf Q \in \mathcal M$ and $\bm v \in \mathbb R^n$.
\begin{align}
   \bm v \trsp \bf Q \bm v^{\langle p - 1 \rangle} &= \sum_{i=1}^n q_{ii}|v^{i}|^p + \sum_{i<j}q_{ij}(v^{i}|v^{j}|^{p-1}\sg(v^{j}) + v^{j}|v^{i}|^{p-1}\sg(v^{i})) \nonumber\\
   &\ge  \sum_{i=1}^n q_{ii}|v^{i}|^p - \sum_{i<j}|q_{ij}|(|v^{i}||v^{j}|^{p-1} + |v^{j}||v^{i}|^{p-1}) \nonumber\\
   &\ge \sum_{i=1}^n q_{ii}|v^{i}|^p - \sum_{i<j}|q_{ij}|(|v^{i}|^p + |v^{j}|^p)\nonumber \\
   &= \sum_{i=1}^n |v^{i}|^p \left(q_{ii} - \sum_{j\neq i}|q_{ij}|\right),
\end{align}
where we used the inequality $x^{p}+y^{p}\ge x^{p-1}y+y^{p-1}x$ for any $p\ge 1$ and $x,y\ge 0$\footnote{To see this, notice that for any $p\ge 1$ and $x,y\ge 0$, $x^{p}+y^{p}-x^{p-1}y-y^{p-1}x=(x^{p-1}-y^{p-1})(x-y)\ge 0$.} to get the third line from the second line above.
Hence the uniform $p$-PD of $\mathcal M$ follows from the item ii) of \thref{thm:equivalent-unfmp+}.
The proof is complete.
\end{proof-of}

\begin{proof-of}[\thref{rmk:1pd=udd}]Suppose $\mathcal M$ is uniform 1-PD. By the item i) of \thref{thm:equivalent-unfmp+}, there exists $\delta, L>0$ such that $\|\bf I - t \bf Q\|_1 \le 1- Lt$ for all $t\in[0,\delta)$ and $\bf Q \in \mathcal M$. Let $\bf Q = (q_{ij})_{n\times n}$ and notice that
\begin{equation}
 \|\bf I - t \bf Q\|_1 = \max_{i\in[n]} \left( |1-tq_{ii}| + \sum_{j\in[n]\setminus\{ i\}}t|q_{ij}|  \right).
\end{equation}
It follows that
\begin{equation}
    \min_{i\in[n] }\left( q_{ii} - \sum_{j\in[n] \setminus \{ i \}} |q_{ij}| \right) \ge L>0.
\end{equation}
Hence $\mathcal M$ is uniformly diagonally dominant (with non-negative diagonal). 
The proof is complete.
\end{proof-of}

\begin{proof-of}[\thref{rmk:ppd->2pd}]Suppose $\mathcal M$ is uniformly $p$-PD but not uniformly $2$-PD. Then, there exists a sequence $\{\bf Q_1, \bf Q_2, \ldots\}\subseteq\mathcal M$ such that the smallest eigenvalues $\lambda_m$ of $\bf Q_m$ satisfy
\begin{equation}\label{eqn:lim-eigvals<=0}
    \lim_{m\to\infty}\lambda_m \le 0.
\end{equation}
For each $m\in\mathbb N^+$, there exists an $\bm v_m \in \mathbb R^n \setminus\{0\}$ such that $\bf Q_m \bm v_m = \lambda_m \bm v_m$. Hence
\begin{equation}
    \bm v_m \trsp \bf Q_m \bm v_m^{\langle p - 1 \rangle} = \lambda_m \bm v_m \trsp \bm v_m^{\langle p - 1 \rangle} = \lambda_m \|\bm v _m\|_p^p.
\end{equation}
But by the item ii) of \thref{thm:equivalent-unfmp+}, there exists $\lambda > 0$ such that $\lambda_m \ge \lambda$. This contradicts \eqref{eqn:lim-eigvals<=0}.
The proof is complete.
\end{proof-of}

\begin{proof-of}[\thref{thm:rate-lp}] We use a technique similar to \cite{krasulina1969stochastic}. Define the function 
\begin{equation}\label{defn:T:t}
\bm T_t(\bm x) = \left(T_t^1(\bm x), \ldots, T_t^n(\bm x)\right)\trsp = \bm x - \bm x^* - \gamma_{t+1} \bm R(x).
\end{equation}

An $n$-dimensional (and corrected) version of the first inequality in
the proof of \citet[Theorem~2]{krasulina1969stochastic} can be obtained by applying \thref{lem:vecexpandp} to our stochastic approximation scheme,
\begin{align}
    \left\|\bm x_{t+1} - \bm x^*\right\|^p_p &= \left\|\bm T_t(\bm x_t) - \gamma_{t+1} \bm \xi_{t+1} \right\|^p_p \\
    &\le \left\|\bm T_t(\bm x_t) \right\|^p_p + 4\gamma_{t+1}^p\left\|\bm \xi_{t+1}\right\|^p_p + p\gamma_{t+1} \sum_{i=1}^n \xi^i_{t+1}\left|T^i_t(\bm x_t)\right|^{p-1} \sg T^i_t(\bm x_t).\label{take:expectation:x:y}
\end{align}
Since $\Exp\left[\xi^i_{t+1}|T^i_t(\bm x_t)|^{p-1} \sg T^i_t(\bm x_t) \mid \bm x_t\right] =|T^i_t(\bm x_t)|^{p-1} \sg T^i_t(\bm x_t) \, \Exp[\xi^i_{t+1} \mid \bm x_t] = 0$, by taking expectations in \eqref{take:expectation:x:y}, we get
\begin{align}
    \Exp \left[  \left\|\bm \delta_{t+1}\right\|^p_p \right] &= \Exp \left[  \left\|\bm x_{t+1} - \bm x^*\right\|^p_p \right]\nonumber \\
    &\le \Exp\left[ \left\|\bm T_t(\bm x_t)\right\|^p_p \right] + 4\gamma_{t+1}^p \Exp \left[\left\|\bm \xi_{t+1}\right\|^p_p \right]\nonumber\\
    &= \Exp\left[\left\|(\bm x_t - \bm x^*) - \gamma_{t+1}\bm R(\bm x_t) \right\|^p_p\right] + 4\gamma_{t+1}^p \Exp \left[\left\|\bm \xi_{t+1}\right\|^p_p \right].
\end{align}
By the mean value theorem, there exists $\bm x^\flat_t \in \{ \bm x^* + \tau(\bm x_t- \bm x^*) : 0\le \tau \le 1  \}$, such that $\bm R(\bm x_t) = \nabla \bm R(\bm x^\flat_t)(\bm x_t - \bm x^*)$, and then
\begin{align}
    & \Exp\left[\left\|(\bm x_t - \bm x^*) - \gamma_{t+1}\bm R(\bm x_t) \right\|^p_p\right] + 4\gamma_{t+1}^p \Exp \left[\left\|\bm \xi_{t+1}\right\|^p_p \right]\nonumber\\
    &= \Exp\left[\left\|(\bf I - \gamma_{t+1} \nabla \bm R(\bm x^\flat_t)) (\bm x_t - \bm x^*)\right\|^p_p\right] + 4\gamma_{t+1}^p \Exp \left[\left\|\bm \xi_{t+1}\right\|^p_p \right] \nonumber\\
    &\le \left\|\bf I - \gamma_{t+1} \nabla \bm R(\bm x^\flat_t) \right\|_p^p \cdot \Exp \left[ \| \bm x_t - \bm x^* \|_p^p \right] + 4\gamma_{t+1}^p \Exp \left[\left\|\bm \xi_{t+1}\right\|^p_p \right] \nonumber \\
    &\le \left\|\bf I - \gamma_{t+1} \nabla \bm R(\bm x^\flat_t) \right\|_p^p \cdot \Exp \left[ \| \bm \delta_t \|_p^p \right] + C_0\gamma_{t+1}^p\left(1 + \Exp\left[\| \bm\delta_t \|_p^p\right]\right),
\end{align}
where the last inequality follows from 
\begin{align}
    \Exp\left[ \left|\bm m_{t+1}\right|^p  \mid  \mathcal F_t\right] &\le  \Exp\left[ \left|\bm m_{t+1}\right|^2  \mid  \mathcal F_t\right]^{p/2} \le  \left[K\left(1+|\bx_t|^2\right)\right]^{p/2}
    \label{eqn:p-state-bound}
    \\ &\le  K^{p/2}\left(1+|\bx_t|^p\right)
    \le K^{p/2}\left(1+2^{p-1}\left(|\bm \delta_t|^p + |\bx^*|^p\right)\right),\nonumber
\end{align}
where we used the inequality $(x+y)^{r}\le x^{r}+y^{r}$ for any $x,y\ge 0$, $0\le r \le 1$
to obtain the first inequality in the second line above, as well as the assumption $\Exp[|\bm \zeta_1|^p]<\infty$. 

\par Note that $\left\|\bf I-\gamma_{t+1}\nabla \bm R(\bm x^\flat_t) \right\|^p_p$ can be estimated by the uniform $p$-PD assumption (see item i) of \thref{thm:equivalent-unfmp+}) since $\gamma_t \to 0$. For $t$ sufficiently large,
\begin{equation}
    \left\|\bf I-\gamma_{t+1}\nabla \bm R(\bm x^\flat_t) \right\|^p_p \le 1-L\gamma_{t+1}.
\end{equation}
And there is a positive constant $C_1$ such that $1-L\gamma_{t+1} + C_0 \gamma_{t+1}^p \le 1- C_1\gamma_{t+1}$ for $t$ sufficiently large. Hence, we arrive at the following iterative bound
\begin{align}\label{iterbound2}
    \Exp\left[  \left\|\bm \delta_{t+1}\right\|^p_p\right] \le (1- \gamma_{t+1}C_1) \cdot \Exp\left[ \left\|\bm \delta_t \right\|^p_p \right] + C_0\gamma_{t+1}^p
\end{align}
for $t$ sufficiently large.

\par Next, let us substitute $\gamma_{t+1}$ with $t^{-\rho}$ where $0 < \rho < 1$. Consider the iteration
\begin{align}\label{mu:iterate}
    \mu_{t+1} = (1- t^{-\rho} C_1) \cdot \mu_t  + C_0 t ^{-\rho p},
\end{align}
so that by \eqref{iterbound2}, $\Exp\left[ \left\|\bm \delta_t \right\|^p_p \right] = \mathcal O(\mu_t)$. By virtue of \thref{lem:chung-fabian}, we get 
\begin{equation}\label{mu:rate}
\mu_t = \Theta \left(t^{-\rho(p-1)}\right). 
\end{equation}
Therefore, by \eqref{iterbound2}, \eqref{mu:iterate}, and \eqref{mu:rate}, 
we obtain the following rate of convergence:
\begin{equation}
   \Exp[\|\bm \delta_t\|^p_p] = \mathcal O \left(t^{-\rho(p-1)}\right).
\end{equation}
Next, since $p$-norms on $\mathbb R^n$ are all equivalent, we can drop the subscript $\|\cdot\|_p$ and obtain
\begin{equation}\label{eqn:rate}
   \Exp[|\bm \delta_t|^p] = \mathcal O \left(t^{-\rho(p-1)}\right).
\end{equation}
Finally, by \eqref{eqn:p-state-bound}, we see that $\sup_{t\in\mathbb N^+}\Exp[|\bm \xi_t |^p] \le \sup_{t\in\mathbb N^+}\Exp[2^{p-1}(|\bm m_t |^p + |\bm \zeta_t|^p)] < \infty$. The proof is complete.
\end{proof-of}
\begin{proof-of}[\thref{thm:rate-lpalpha}]Under the assumptions of \thref{thm:rate-lpalpha}, the rate $\Exp[|\bm \delta_t|^p] = \mathcal{O}\left(t^{ -\rho (p-1)  }\right)$ holds for every $p\in [q, \alpha)$. We can thus apply Jensen's inequality to strengthen it. By Jensen's inequality and \eqref{eqn:lpweak}, we get
\begin{equation}
   \Exp\left[|\bm \delta_t|^q\right] \le \Exp[ |\bm \delta_t|^p ]^{q/p} = \mathcal{O}\left(t^{ -\rho (p-1)\frac{q}{p}  }  \right).
\end{equation}
By letting $p \nearrow \alpha$, we conclude that have for every $\varepsilon > 0$,
\begin{equation}
    \Exp\left[|\bm \delta_t|^q\right] = o \left(t^{-\rho q \frac{\alpha - 1}{\alpha} + \varepsilon}  \right).
\end{equation}
The proof is complete.  %
\end{proof-of}

Next, we state a series of technical lemmas as well as their proofs, which will be used in the proofs of Theorems~\ref{thm:stab-linear}
and \ref{thm:stab-nonlin}.

\begin{lemma}\thlabel{lem:rho-exp} If $\gamma_t \asymp  t^{-\rho}$ with $0 < \rho < \kappa \le 1$, then for all $\lambda > 0$, 
    \begin{align}\label{second:limit}
\lim_{t \to \infty} t^{-\kappa} \sum_{j=1}^{t-1}\exp\left(-\lambda\sum_{i=j}^{t-1} \gamma_i \right) = 0.
    \end{align}
\end{lemma}

\begin{proof}
Notice that there exists some constant $B>0$ such that
\begin{equation}
     \sum_{i=j}^{t-1} \gamma_i\ge \frac B \lambda \left(t^{1-\rho}- j^{1-\rho} \right).
\end{equation}
It follows that
\begin{equation}
t^{-\kappa} \sum_{j=1}^{t-1}\exp\left(-\lambda\sum_{i=j}^{t-1} \gamma_i \right)\le t^{-\kappa} \sum_{j=0}^{t-1}\exp\left(-Bt^{1-\rho}+Bj^{1-\rho} \right)
= \frac{\sum_{j=0}^{t-1} \exp(Bj^{1-\rho})}{t^{\kappa}\exp(Bt^{1-\rho})}.
\end{equation}
By Stolz-Ces\`aro theorem, we have
\begin{align*}
\frac{\sum_{j=0}^{t-1} \exp(Bj^{1-\rho})}{t^{\kappa}\exp(Bt^{1-\rho})}
&\asymp \frac{\exp(Bt^{1-\rho})} {(t+1)^{\kappa}\exp(B(t+1)^{1-\rho}) - t^{\kappa}\exp(Bt^{1-\rho})}  \\
&= \frac{1}{(t+1)^{\kappa}\exp[B((t+1)^{1-\rho}-t^{1-\rho})] - t^{\kappa}} \\
&= \frac{1}{(t+1)^{\kappa}\exp[B(1-\rho)(t+1)^{-\rho} + o(t^{-\rho})] - t^{\kappa}}  \\
&= \frac{1}{(t+1)^{\kappa}[1 + B(1-\rho)(t+1)^{-\rho} + o(t^{-\rho})] - t^{\kappa}} \\
&= \frac{1}{B(1-\rho)(t+1)^{\kappa-\rho} + o((t+1)^{\kappa-\rho})} \\
&\to 0,
\end{align*}
as $t\rightarrow\infty$. 
The proof is complete.
\end{proof}

\begin{lemma}\thlabel{lemkappa}
Suppose $\gamma_t \asymp  t^{-\rho}$ and $0 < \rho < \kappa \le 1$; let $\bf A$ be a positive definite symmetric matrix. Consider the matrix recursion in \cite[Lemma~1]{polyak1992acceleration},
\begin{align}
    \bf X_{j}^j  = \bf I, \quad \bf X_j^{t+1} = \bf X_j^t - \gamma_t \bf A \bf X_j^t, \quad (j\in\mathbb N^+)
\end{align}
and define
\begin{align}
       \overline{\bf X}_j^t = \gamma_j \sum_{i=j}^{t-1}\bf X_j^i, \quad \bf \Phi_j^t = \bf A^{-1} - \overline{\bf X}_j^t. 
\end{align}
Then the following limit holds,
\begin{align}
\lim_{t\to\infty} \frac 1 {t^{\kappa}} \sum_{j=1}^{t-1} \|\bm \Phi_j^t \| = 0.
\end{align}
\end{lemma}

\begin{remark}
Lemma~\ref{lemkappa} recovers \cite[Lemma~1]{polyak1992acceleration} as the special case $\kappa = 1$.
\end{remark}

\begin{proof-of}[\thref{lemkappa}]
Modeling after \cite{polyak1992acceleration}'s proof of their Lemma 1, we define $\bf S_j^t = \sum_{i=j}^{t-1} (\gamma_i-\gamma_j)\bf X_j^i$, and we have
\begin{align}
\bm \Phi_j^t = \bf S_j^t + \bf A^{-1} \bf X_j^t.
\end{align}
We will split the proofs into two parts. In the first part, we will prove $t^{-\kappa} \sum_{j=1}^{t-1} \|\bf S_j^t \|\to 0$ and 
then in the second part we will prove $t^{-\kappa} \sum_{j=1}^{t-1} \|\bf X_j^t \| \to 0$.

\textbf{Part I.} We first prove that $t^{-\kappa} \sum_{j=1}^{t-1} \|\bf S_j^t \|\to 0$.

\par By the Part 3 of \citet[Lemma~1]{polyak1992acceleration}\footnote{We can directly use this inequality since our assumption on step-size $\gamma_t \asymp t^{-\rho}$, $0 < \rho < 1$ can meet \citet[Assumption~2.2]{polyak1992acceleration}.}, there exist some $\lambda > 0$ and $K < \infty$ such that
\begin{equation}\label{boundXexp}
    \|\bf X_j^t \| \le K \exp\left(-2\lambda\sum_{i=j}^{t-1} \gamma_i \right) = Ke^{-2\lambda m_j^t}, 
\end{equation}
where $m_{k}^{\ell}$ stands for $\sum_{i = k}^{\ell-1} \gamma_i$. Now we have
\begin{align}
   \left\|\bf S_j^t \right\| &= \left\| \sum_{i=1}^t (\gamma_i - \gamma_j) \bf X_j^i  \right \|  \nonumber\\
   &= \left\| \sum_{i=1}^t\left[\sum_{k=j}^{i-1} (\gamma_{k+1}-\gamma_k)\right] \bf X_j^i  \right \| \nonumber \\
   &\le C_0 \sum_{i=j}^t \sum_{k=j}^{i-1} k^{-\rho-1}  \exp\left(-2\lambda m^i_j\right)  \nonumber  \\
   &\le C_0 j^{-1} \sum_{i=j}^t \sum_{k=j}^{i-1} k^{-\rho}  \exp\left(-2\lambda m^i_j\right)   \nonumber \\
   &\le C_1 j^{-1} \sum_{i=j}^t m^i_j  \exp\left(-2\lambda m^i_j\right)   \nonumber \\
   &= C_1 j^{-1} \sum_{i=j}^t\frac{m_j^i e^{-2\lambda m_j^i} (m_j^i - m_j^{i-1})}{\gamma_i},\label{integralestimation}
\end{align}
where $C_{0},C_{1}$ are some positive constants.

Since the function $f_w(x) = x^\rho \exp(-w x^{1-\rho})$ is bounded on $x\in [1, \infty )$ for every $w > 0$, we have
\begin{align}
    \frac{j^{-\rho}}{\gamma_i}\exp\left(-\lambda m_j^i\right) \le C_2 i^\rho j^{-\rho} \exp(-C_3(i^{1-\rho} - j^{1-\rho})) = C_2 f_{C_3}(i)/f_{C_3}(j) \le C_4,
\end{align}
for some positive constants $C_{2}$, $C_3$ and $C_4$.
Hence, continuing upon \eqref{integralestimation},
\begin{align}
     \left\|\bf S_j^t \right\| \le C_1 C_4 j^{\rho-1}  \sum_{i=j}^t m_j^i e^{-\lambda m_j^i} (m_j^i - m_j^{i-1}).
\end{align}
Since the summation $ \sum_{i=j}^t m_j^i e^{-\lambda m_j^i} (m_j^i - m_j^{i-1})$ approximates $\int_{0}^{m_{j}^{t}} m e^{-\lambda m} \d m$, it is bounded. Hence, for some positive constant $C_{5}$,
\begin{align}
     \|\bf S_j^t \| \le C_5 j^{\rho-1},
\end{align}
which implies the desired limit
\begin{align}
\lim_{t \to \infty}    t^{-\kappa} \sum_{j=1}^{t-1} \|\bf S_j^t \| = 0.
\end{align}

\textbf{Part II.} It remains to prove that $t^{-\kappa} \sum_{j=1}^{t-1} \|\bf X_j^t \| \to 0$.

\par Combining \eqref{boundXexp} and \thref{lem:rho-exp}, we have $t^{-\kappa} \sum_{j=1}^{t-1} \|\bf X_j^t \| \to 0$. Hence the proof of this lemma is complete.
\end{proof-of}

\begin{lemma}\thlabel{lemneveu}
Given the assumption of \thref{thm:stab-linear} or \thref{thm:stab-nonlin},
\begin{align}
    \frac{\bm \xi_1 + \ldots \bm \xi_t}{t^{1/\alpha}} \xrightarrow[t\to \infty]{\mathcal D} \mu.
\end{align}
\end{lemma}
\begin{proof} We recall the decomposition $\bm \xi_t = \bm \zeta_t +  \bm m_t$, where $\{ \bm \zeta_t \}$ are i.i.d. and $\bm \zeta_1$ is in the domain of normal attraction of an $n$-dimensional centered $\alpha$-stable distribution so that
$$ \frac{\bm \zeta_1 +\ldots+ \bm \zeta_t}{t^{1/\alpha}} \xrightarrow[t\to \infty]{\mathcal D} \mu.$$
Hence, 
it suffices to show that $t^{-1/\alpha}(\bm m_1 + \ldots + \bm m_t) \to 0$ in $L^r$, for some $r\ge 1$.

\par By \eqref{eqn:state-dep-var}, there exists a constant $C>0$ such that
\begin{equation}
     \Exp\left[ \left|\bm m_{t+1}(\bx_t)\right|^2  \mid  \mathcal F_t\right] \le K\left(1+|\bx_t|^2\right) \le K(1+2|\bx^*|^2 + 2|\bm \delta_t|^2) \le C(1+|\bm \delta_t|^2).
\end{equation}

\par Hence, by using the ``Remark'' on p.151 of \cite{neveu1975discrete} (cf. inequalities (20) of \cite{anantharam2012stochastic}), we get
\begin{align}
    \Exp\left[\left | \frac{\bm m_1 + \ldots + \bm m_t}{t^{1/\alpha}} \right |^r\right] &\le \frac{C_1}{t^{r/\alpha}} \Exp\left[\left(\sum_{i=1}^t \Exp\left[ |\bm m_i|^2 \mid \mathcal F_{i-1} \right]\right) ^{r/2}  \right] \nonumber\\
    &\le \frac{C_2}{t^{r/\alpha}} \Exp\left[  \left( \sum_{i=1}^t \left(1 + |\bm \delta_{i-1}|^2\right)  \right)^{r/2}  \right] \nonumber\\
    &\le \frac{C_2}{t^{r/\alpha}} \Exp\left[t^{r/2} + \sum_{i=1}^t |\bm\delta_{i-1}|^r \right],\label{ineq:remark}
\end{align}
where, for the last inequality, we use the fact that $(x+y)^{s}\le x^{s}+y^{s}$ for any $x,y\ge 0$, $0\le s \le 1$.
\par If the assumption of \thref{thm:stab-linear} holds, take $r = p > (\alpha + \alpha \rho)/(1 + \alpha \rho)$ in the inequalities \eqref{ineq:remark} above. Then, by \thref{thm:rate-lp}, $\Exp[|\bm \delta_t|^r] = \mathcal O (t^{-\rho (r-1)}) = o(t^{r/\alpha - 1})$.

\par If the assumption of \thref{thm:stab-nonlin} holds, take $r = q>1/\rho > \alpha / (1+\rho(\alpha - 1))$ in the inequalities \eqref{ineq:remark} above. Then by \thref{thm:rate-lpalpha}, $\Exp[|\bm \delta_t|^r] =\tilde{\mathcal{O}} (t^{-\rho r (\alpha-1)/\alpha}) = o(t^{r/\alpha - 1})$.

\par In both cases, $t^{-1/\alpha}(\bm m_1 + \ldots + \bm m_t) \to 0$ in $L^r$. The proof is complete.
\end{proof}

Finally, we are ready to prove Theorems~\ref{thm:stab-linear}
and \ref{thm:stab-nonlin}.

\begin{proof-of}[\thref{thm:stab-linear}]By \citet[Lemma 2]{polyak1992acceleration}:%
\begin{equation}\label{eqn:error-parts}
    \frac{t}{t^{1/\alpha}} \overline{\bm \delta}_t = \underbrace{\frac{1}{t^{1/\alpha}}\bf F_t \bm\delta_0}_{\bm I^{(1)}_t} - \underbrace{\frac{1}{t^{1/\alpha}} \sum_{j=1}^{t-1} \bf A^{-1} \bm\xi_j}_{\bm I^{(2)}_t} - \underbrace{\frac{1}{t^{1/\alpha}} \sum_{j=1}^{t-1} \bf W_j^t \bm\xi_j}_{\bm I^{(3)}_t},
\end{equation}
where $\bf F_t$ and $\bf W_j^t$ are deterministic matrices with uniformly bounded operator 2-norms defined as
\begin{align}
\bf F_t &= \sum_{i=0}^{t-1}\prod_{k=1}^i(\bf I-\gamma_k \bf A), \label{eqn:def-F} \\
\bf W_j^t &= \gamma_j \sum_{i=j}^{t-1}\prod_{k=j+1}^i(\bf I-\gamma_k \bf A) - \bf A^{-1}. \label{eqn:def-W}
\end{align}
We have $\bm I^{(1)}_t \to 0$ by the boundedness of $\bf F_t$. 
Next, take some $\kappa$ such that
\begin{equation}\label{eqn:take-kappa}
    \max(\rho, 1/\alpha) < \kappa \le p/\alpha.
\end{equation}
We shall prove that $\bm I^{(3)}_t \to 0$ in $L^{\alpha\kappa}$ (notice that $1< \alpha\kappa \le p < \alpha$; cf. \citet[Proof of Theorem~1]{polyak1992acceleration} where convergence in $L^2$ is proven). By \thref{thm:rate-lp}, $\sup_j \Exp[|\bm \xi_j|^p] < \infty$. Hence we can compute, by virtue of \thref{lem:p-expand}, that
\begin{align} 
\Exp\left[ \left| \bm I^{(3)}_t \right|^{\alpha\kappa} \right] &=
\Exp\left[\left |  \frac{1}{t^{1/\alpha}} \sum_{j=1}^{t-1} \bf W_j^t \bm\xi_j \right |^{\alpha\kappa} \right] \le \frac{C_0}{t^{\kappa}}  \sum_{j=1}^{t-1}\Exp\left[\left | \bf W_j^t \bm\xi_j \right |^{\alpha\kappa} \right] \nonumber\\
&\le \left(\frac{C_0}{t^{\kappa}}\sum_{j=1}^{t-1} \left\|\bf W_j^t \right\|^{\alpha\kappa}\right) \, \sup_{j} \Exp\left[\left|\bm \xi_j\right|^{\alpha\kappa}\right] \le \left(\frac{C_0}{t^{\kappa}}\sum_{j=1}^{t-1} \left\|\bf W_j^t \right\|\right) \, \sup_{j} \Exp\left[\left|\bm \xi_j\right|^{\alpha\kappa}\right] \nonumber\\
&\le \frac{C_1}{t^{\kappa}}\sum_{j=1}^{t-1} \left\|\bf W_j^t \right\|.
\end{align}
Notice that the matrices $\bf W_j^t$ defined above correspond to $-\bm \Phi_j^t$ in \thref{lemkappa}. This infers that $\Exp\left[ | \bm I^{(3)}_t |^{\alpha\kappa} \right] \le \frac{K_1}{t^{\kappa}}\sum_{j=1}^{t-1} \|\bf W_j^t \| \to 0$ as $t \to \infty$.
 
\par Finally, \thref{lemneveu} states that $\bm I^{(2)}_t$ converges weakly to an $\alpha$-stable distribution. Hence we conclude the proof.
\end{proof-of}

\begin{proof-of}[\thref{thm:stab-nonlin}]Denote by $\bf A$ the Hessian matrix $\grad \bm R(\bx^*) = \grad^2 f(\bx^*)$. Consider a corresponding linear SA process with the same noise,
\begin{equation}\label{eqn:linear-SA-same-noise}
   \bx_{t+1}^1 = \bx_{t}^1 - \step_{t+1} \left( \bf A(\bx_{t}^1 - \bx^*) + \bxi_{t+1}(\bx_t)\right),
\end{equation}
with $\bx_0 ^1 = \bx_0$. We further define $\bm \delta_t^1 = \bx_t ^1 - \bx^*$ and the averaging process $\overline{\bm \delta}_t^1 = (\bm \delta_0^1 + \ldots + \bm \delta_{t-1}^1)/t$.

\par \textbf{Part I.} We first prove that $t^{1-1/\alpha}\left(\bm {\overline\delta}^1_t - \bm {\overline\delta}_t \right) \to 0$ almost surely.

\par By \eqref{eqn:error-parts}, we have
\begin{equation}\label{eqn:delta1-avg}
    \frac{t}{t^{1/\alpha}} \overline{\bm \delta}^1_t = \frac{1}{t^{1/\alpha}}\bf F_t \bm\delta_0 - \frac{1}{t^{1/\alpha}} \sum_{j=1}^{t-1} \left(\bf A^{-1}+ \bf W_j^t\right)\bm\xi_j,
\end{equation}
where the matrices $\bf F_t$ and $\bf W_j^t$ are defined back in \eqref{eqn:def-F} and \eqref{eqn:def-W}. For the non-linear process \eqref{eqn:SA}, it can be viewed \emph{as if it is a linear process with the $j$-th noise term being $\bm \xi_j + \bm R (\bx_{j-1}) - \bf A \bm \delta_{j-1}$}. Hence by \eqref{eqn:error-parts}, we have
\begin{equation}\label{eqn:delta-avg}
    \frac{t}{t^{1/\alpha}} \overline{\bm \delta}_t = \frac{1}{t^{1/\alpha}}\bf F_t \bm\delta_0 - \frac{1}{t^{1/\alpha}} \sum_{j=1}^{t-1} \left(\bf A^{-1}+ \bf W_j^t\right)\left(\bm \xi_j + \bm R (\bx_{j-1}) - \bf A \bm \delta_{j-1}\right).
\end{equation}
Combining \eqref{eqn:delta1-avg} and \eqref{eqn:delta-avg} yields the difference (cf. Part 4 of \citet[Proof of Theorem~2]{polyak1992acceleration})
\begin{equation}\label{eqn:def-lin-nonlin}
 \frac{t}{t^{1/\alpha}} \left(\bm {\overline\delta}^1_t - \bm {\overline\delta}_t \right) = \frac{1}{t^{1/\alpha}}\sum_{j=1}^{t-1}\left(\bf A^{-1} + \bf W^t_j\right)\left(\bm R (\bm x_{j-1}) - \bf A \bm \delta _{j-1}\right).
\end{equation}
We also recall the assumption that $|\bm R (\bm x_j) - \bf A \bm \delta _j| \le K|\bm \delta_j|^q$. Hence, it suffices to show the following term vanishes
almost surely as $t\rightarrow\infty$:
\begin{align}
J_t = \frac{1}{t^{1/\alpha}}\sum_{j=1}^{t-1}|\bm \delta_j|^{q}.
\end{align}
To show this, first by our calculation of the rate of convergence in \thref{thm:rate-lpalpha},
\begin{equation}
    \Exp\left[\sum_{j=1}^{t-1} \frac{1}{j^{1/\alpha}}|\bm \delta_j|^{q}  \right] = \sum_{j=1}^{t-1} \tilde{\mathcal{O}}\left(j^{-\rho q\frac{\alpha - 1}{\alpha} - \frac{1}{\alpha}}\right) = \mathcal{O}(1).
\end{equation}
The last equality holds since $-\rho q\frac{\alpha - 1}{\alpha} - \frac{1}{\alpha} < -1$.
Hence, we have
\begin{equation}\label{eqn:kron}
    \Pr\left[\sum_{j=1}^{t-1} \frac{1}{j^{1/\alpha}}|\bm \delta_j|^{q} < \infty  \right] = 1.
\end{equation}
By Kronecker's lemma, \eqref{eqn:kron} implies that $\Pr[\lim_{t\to\infty}J_t = 0] = 1$. This further implies that the left hand side of \eqref{eqn:def-lin-nonlin}, $t^{1-1/\alpha}\left(\bm {\overline\delta}^1_t - \bm {\overline\delta}_t \right)$, converges to 0 almost surely.

\par \textbf{Part II.} It remains to show that $t^{1-1/\alpha}\bm {\overline\delta}^1_t$ converges weakly to an $\alpha$-stable distribution.

\par Define $\bm {\overline x}_t^1 = (\bx_0^1 + \ldots + \bx_{t-1}^1)/t$. Since $t^{1-1/\alpha}\left(\bm {\overline x}^1_t - \bm {\overline x}_t \right) = t^{1-1/\alpha}\left(\bm {\overline\delta}^1_t - \bm {\overline\delta}_t \right) \to 0$ almost surely, it follows \emph{a fortiori} that $\bm {\overline x}^1_t - \bm {\overline x}_t \to 0$ almost surely. Hence $\bx_t ^1 - \bx_t \to 0$ almost surely, due to the well-known theorem that a real-valued sequence converges to zero if and only if the average sequence converges to zero.

\par Therefore, for the noise decomposition $\bxi_{t+1} (\bx_t) = \bm \zeta_{t+1} + \bm m_{t+1}(\bx_t)$, the state-dependent component $\bm m_{t+1}(\bx_t)$ satisfies not only \eqref{eqn:state-dep-var}, i.e.,
 \begin{equation}
     \Exp\left[ \left|\bm m_{t+1}(\bx_t)\right|^2  \mid  \mathcal F_t\right] \le K\left(1+|\bx_t|^2\right),
 \end{equation}
 but also
 \begin{equation}
     \Exp\left[ \left|\bm m_{t+1}(\bx_t)\right|^2  \mid  \mathcal F_t\right] \le C\left(1+|\bx_t ^1|^2\right).
 \end{equation}
Hence, combining the discussion above and \thref{lemneveu}, we know that the linear recursion \eqref{eqn:linear-SA-same-noise} defines a process that satisfies \thref{thm:stab-linear}. (The only difference is that $\kappa$, instead of \eqref{eqn:take-kappa}, can be taken from the range $ (\rho,1)$ under the assumption of the current theorem, since by \thref{thm:rate-lp}, $\sup_{t\in\mathbb N^+}\Exp[|\bm \xi _t|^p] < \infty$ for every $1\le p<\alpha$.) It then follows from \thref{thm:stab-linear} that $t^{1-1/\alpha}\bm {\overline\delta}^1_t$ converges weakly to an $\alpha$-stable distribution. 

\par The proof is complete.
\end{proof-of}

\section{Additional Technical Background}\label{sec:more-prelim}

\subsection{Properties of $\alpha$-Stable Distributions}
An $\alpha$-stable distributed random variable $X$ is 
denoted by 
$X \sim \mathcal{S}_{\alpha}(\sigma,\theta,\mu)$, where $\alpha\in(0,2]$ is the \emph{tail-index}, $\theta\in[-1,1]$ is the \emph{skewness} parameter, $\sigma\ge 0$ is the \emph{scale} parameter, and $\mu\in\mathbb{R}$ is called the \emph{location} parameter. An $\alpha$-stable random variable $X$
is uniquely characterized by its 
characteristic function:
$\Exp\left[\exp\left(iu X\right)\right] 
=
e^{-\sigma^{\alpha}|u|^{\alpha}(1-i\theta\text{sgn}(u)\tan(\frac{\pi\alpha}{2}))+i\mu u}$, if $\alpha\neq 1$
and 
$\Exp\left[\exp\left(iu X\right)\right] 
=e^{-\sigma|u|(1+i\theta\frac{2}{\pi}\text{sgn}(u)\log|u|)+i\mu u}$,
if $\alpha=1$,
for any $u\in\mathbb{R}$.
The mean of $X$ coincides with $\mu$ if $\alpha>1$,
and otherwise the mean of $X$ is undefined.
The skewness parameter $\theta$ is a measure of asymmetry. We say that $X$ follows a \emph{symmetric} $\alpha$-stable distribution
denoted as
$\mathcal{S}\alpha\mathcal{S}(\sigma)=\mathcal{S}_{\alpha}(\sigma,0,0)$ 
if $\theta=0$ (and $\mu=0$). 
The tail-index parameter $\alpha\in(0,2]$ determines
the tail thickness of the distribution, 
and $\sigma>0$ measures the spread
of $X$ around its mode. 
When $\alpha<2$, $\alpha$-stable distributions 
have heavy tails so that their moments
are finite only up to the order $\alpha$.
More precisely, let $X\sim\mathcal{S}_{\alpha}(\sigma,\theta,\mu)$
with $0<\alpha<2$.
Then $\Exp[|X|^{p}]<\infty$ for any $0<p<\alpha$
and $\Exp[|X|^{p}]=\infty$
for any $p\ge\alpha$,
which implies infinite variance (see e.g. \cite[Property~1.2.16]{ST1994}).
When $0<\alpha<2$, the left tail and right tail of $X$
are described by the formulas:
\begin{align}
\lim_{x\rightarrow\infty}x^{\alpha}\Pr(X>x)
=\frac{1+\theta}{2}C_{\alpha}\sigma^{\alpha},
\qquad
\lim_{x\rightarrow\infty}x^{\alpha}\Pr(X<-x)
=\frac{1-\theta}{2}C_{\alpha}\sigma^{\alpha},
\end{align}
where $C_{\alpha}:=(1-\alpha)/(\Gamma(2-\alpha)\cos(\pi\alpha/2))$
if $\alpha\neq 1$ and $C_{\alpha}:=2/\pi$ if $\alpha=1$,
(see e.g. \cite[Property~1.2.15]{ST1994}).
The family of $\alpha$-stable distributions include
normal, L\'{e}vy and Cauchy distributions as special cases, 
and can be used to model many complex stochastic phenomena \citep{Sarafrazi2019, Fiche2013, Farsad2015}.

\subsection{Domains of Attraction of Stable Distributions}
 
Let $X_{i}$ be an i.i.d. sequence with a common distribution whose distribution function
is denoted as $F$, and let $S_{n}:=X_{1}+X_{2}+\cdots+X_{n}$. Suppose that for some normalizing constants
$a_{n}>0$ and $b_{n}$, the sequence $S_{n}/a_{n}-b_{n}$ has a non-degenerate limit distribution with distribution function $G$,
i.e. 
\begin{equation}\label{eqn:relation}
\lim_{n\rightarrow\infty}\Pr(S_{n}/a_{n}-b_{n}\le x)=G(x),
\end{equation}
for all continuity points $x$ of $G$,
then such limit distributions $G$ are stable distributions and the set of distribution functions $F$
such that $S_{n}/a_{n}-b_{n}$ converges to a particular stable distribution is called its \emph{domain of attraction}.

Next, let us provide a sufficient and necessary condition for being in the domain of attraction
of a stable distribution. 
The class of distribution functions $F$ for which $S_{n}/a_{n}-b_{n}$ converges to $\mathcal{S}\alpha\mathcal{S}(\sigma)$
is called the $\alpha$-stable domain of attraction, and we denote it as $F\in D_{\alpha}$. Before we proceed, let us recall that a positive measurable function $f$ is \emph{regularly varying} if there exists a constant $\gamma\in\mathbb{R}$
such that
$\lim_{t\rightarrow\infty}\frac{f(tx)}{f(t)}=x^{\gamma}$, for every $x>0$.
In this case, we denote $f\in RV_{\gamma}$, and we say a function $f$ is slowly varying
if $f\in RV_{0}$. 

Define the characteristic functions
$\phi(u):=\int_{-\infty}^{\infty}e^{iux}dF(x)$
and
$\psi(u):=\int_{-\infty}^{\infty}e^{iux}dG(x)$,
and also define
$\lambda(u):=\phi(1/u)$ and $g(u):=\psi(1/u)$
for $u\in[-\infty,\infty]\backslash\{0\}$. 
We also denote $U(x):=\text{Re}\lambda(x)$ and $V(x):=\text{Im}\lambda(x)$.
By L\'{e}vy's continuity theorem for characteristic functions (see e.g. \citet[Chapter XV.3]{Feller}),
the convergence in \eqref{eqn:relation} is equivalent to $\lim_{n\rightarrow\infty}\exp(-ib_{n}/u)\lambda^{n}(a_{n}u)=g(u)$,
$u\in[-\infty,\infty]\backslash\{0\}$ uniformly on neighborhoods of $\pm\infty$.
Based on this, one can show that (see e.g. ) if \eqref{eqn:relation} holds, 
then $|g(u)|^{2}=\exp(-c|u|^{-\alpha})$ for some $\alpha\in(0,2]$ and $c>0$
and moreover $-\log|\lambda|\in RV_{-\alpha}$, i.e. $-\log|\lambda|$ is regularly varying
with index $-\alpha$. Next, we state a sufficient and necessary condition for being in the $\alpha$-stable domain of attraction.

\begin{theorem}[\cite{Geluk}, Theorem~1]
Suppose $0<\alpha<2$. Every $\alpha$-stable random variable $X$ has a characteristic function of the form:
\begin{equation}
\Exp\left[\exp\left(iuX\right)\right]=\exp\left(-\left\{|u|^{\alpha}+iu(2p-1)\left\{(1-\alpha)\tan(\alpha\pi/2)\right\}\frac{|u|^{\alpha-1}-1}{\alpha-1}\right\}\right),
\end{equation}
for some $0\le p\le 1$ with $(1-\alpha)\tan(\pi/2)$ defined to be $2/\pi$ at $\alpha=1$. 
The following statements are equivalent:

(i) $F\in D_{\alpha}$.

(ii) $1-F(x)+F(-x)\in RV_{-\alpha}$ and there exists a constant $p\in[0,1]$ such that
\begin{equation}
\lim_{x\rightarrow\infty}\frac{1-F(x)}{1-F(x)+F(-x)}=p.
\end{equation}

(iii) $1-U(x)\in RV_{-\alpha}$ and there exists a constant $p\in[0,1]$ such that
\begin{equation}
\lim_{x\rightarrow\infty}\frac{xuV(xu)-xV(x)}{x(1-U(x))}=(2p-1)(1-\alpha)\tan\left(\frac{\alpha\pi}{2}\right)\frac{|u|^{1-\alpha}-1}{1-\alpha}, \qquad u\in\mathbb{R}\backslash\{0\}.
\end{equation}
\end{theorem}

Furthermore, \cite[Theorem~1]{Geluk} showed that if any of (i), (ii), (iii) holds, then $\lim_{x\rightarrow\infty}\frac{1-U(x)}{1-F(x)+F(-x)}=\Gamma(1-\alpha)\cos(\alpha\pi/2)$
and $\lim_{x\rightarrow\infty}\frac{V(x)-x^{-1}\int_{0}^{x}(1-F(y)-F(-y))\d y}{1-F(x)+F(-x)}=(2p-1)\left(\Gamma(1-\alpha)\sin(\alpha\pi/2)-\frac{1}{1-\alpha}\right)$.

Let us illustrate \cite[Theorem~1]{Geluk} with an example of Pareto distribution, which is a power-law distribution
widely applied in various fields. 
A random variable $X$ is said to follow a Pareto distribution (of type I)
if there exists some $c>0$ such that $\Pr(X>x)=(x/c)^{-\alpha}$
for any $x\ge c$ and $\Pr(X>x)=1$ for any $x<c$. 
In this case, $F(x)=1-(x/c)^{-\alpha}$ for any $x\ge c$
and $F(x)=0$ for any $x<c$. 
It follows that 
$1-F(x)+F(-x)\in RV_{-\alpha}$ and 
$\lim_{x\rightarrow\infty}\frac{1-F(x)}{1-F(x)+F(-x)}=1$.
Therefore, $F\in D_{\alpha}$ and the Pareto distribution 
is in the $\alpha$-stable domain of attraction.

When the tail-index $\alpha\in(0,2)$, the logarithm
of the characteristic function (i.e. $\log\mathbb{E}\left[e^{iuX}\right]$) of an $\alpha$-stable random variable $X$
is of the form (see \cite[equation (12), page 168]{gnedenko1968limit}):
\begin{equation}\label{eqn:characterized}
i\gamma u+c_{1}\int_{-\infty}^{0}\left[e^{iux}-1-\frac{iux}{1+x^{2}}\right]\frac{\d x}{|x|^{1+\alpha}}
+c_{2}\int_{0}^{\infty}\left[e^{iux}-1-\frac{iux}{1+x^{2}}\right]\frac{\d x}{x^{1+\alpha}},
\end{equation}
where $c_{1},c_{2}\ge 0$ and $\gamma\in\mathbb{R}$.
Since the characteristic function uniquely characterizes a probability distribution,
the triplet $(c_{1},c_{2},\alpha)$ uniquely determines an $\alpha$-stable
law up to a constant shift $\gamma\in\mathbb{R}$ when $0<\alpha<2$.
\cite[Theorem~2, page 175]{gnedenko1968limit} provides
another sufficient and necessary condition for being
in the domain of attraction of an $\alpha$-stable distribution,
which complements \cite[Theorem~1]{Geluk}.
Suppose $0<\alpha<2$. 
Then, the distribution function $F(x)$ belongs to the domain
of attraction of an $\alpha$-stable distribution
if and only if the following conditions hold:
(i) $\lim_{x\rightarrow\infty}\frac{F(-x)}{1-F(x)}=\frac{c_{1}}{c_{2}}$.
(ii) For every constant $\kappa>0$, $\lim_{x\rightarrow\infty}\frac{1-F(x)+F(-x)}{1-F(\kappa x)+F(-\kappa x)}=\kappa^{\alpha}$.
In the case of a Pareto distribution (of type I), 
for some $c>0$,
we have $F(x)=1-(x/c)^{-\alpha}$ for any $x\ge c$
and $F(x)=0$ for any $x<c$. 
Then we can check that  $\lim_{x\rightarrow\infty}\frac{F(-x)}{1-F(x)}=0$
and for every constant $\kappa>0$, $\lim_{x\rightarrow\infty}\frac{1-F(x)+F(-x)}{1-F(\kappa x)+F(-\kappa x)}=\lim_{x\rightarrow\infty}\frac{(x/c)^{-\alpha}}{(\kappa x/c)^{-\alpha}}=\kappa^{\alpha}$.
Thus, the Pareto distribution
belongs to the domain
of attraction of an $\alpha$-stable distribution.

Finally, let us provide a sufficient and necessary condition for being in the domain of normal attraction
of a stable distribution. 

\begin{theorem}[\cite{gnedenko1968limit}, Theorem~5, page 181]
Suppose $0<\alpha<2$. The distribution function $F(x)$ belongs to the domain
of attraction of an $\alpha$-stable distribution characterized by \eqref{eqn:characterized}
if and only if 
\begin{align}
&F(x)=(c_{1}a^{\alpha}+\alpha_{1}(x))\frac{1}{|x|^{\alpha}},\qquad\text{for $x<0$},\label{eqn:F:1}
\\
&F(x)=1-(c_{2}a^{\alpha}+\alpha_{2}(x))\frac{1}{x^{\alpha}},\qquad\text{for $x>0$},\label{eqn:F:2}
\end{align}
where $a>0$ is a positive constant and $\alpha_{1}(x),\alpha_{2}(x)$ are functions
satisfying $\lim_{x\rightarrow-\infty}\alpha_{1}(x)=\lim_{x\rightarrow\infty}\alpha_{2}(x)=0$.
Indeed, the constant $a$ in \eqref{eqn:a}, \eqref{eqn:F:1} and \eqref{eqn:F:2} is the same.
\end{theorem}

In the case of a Pareto distribution (of type I), 
for some $c>0$,
we have $F(x)=1-(x/c)^{-\alpha}$ for any $x\ge c$
and $F(x)=0$ for any $x<c$. 
Then we can check that \eqref{eqn:F:1} and \eqref{eqn:F:2} hold
with $c_{1}=0$, $\alpha_{1}(x)\equiv 0$, $c_{2}=1$, $\alpha_{2}(x)\equiv 0$ and $a=c$.
Thus, the Pareto distribution
belongs to the domain
of normal attraction of an $\alpha$-stable distribution.

\end{document}